\documentclass[11pt,reqno]{amsart}
\usepackage[utf8]{inputenc}
\usepackage[T1]{fontenc}
\usepackage{amsmath,amssymb,amsthm,mathtools}
\usepackage[margin=1.15in]{geometry}
\usepackage{xcolor}
\usepackage{booktabs}
\usepackage{array}

\usepackage[expansion=false]{microtype}
\usepackage[colorlinks=true,linkcolor=blue!55!black,citecolor=blue!55!black,urlcolor=blue!55!black]{hyperref}

\theoremstyle{plain}
\newtheorem{theorem}{Theorem}[section]
\newtheorem{proposition}[theorem]{Proposition}
\newtheorem{conjecture}[theorem]{Conjecture}
\newtheorem{lemma}[theorem]{Lemma}
\newtheorem{corollary}[theorem]{Corollary}
\theoremstyle{definition}
\newtheorem{definition}[theorem]{Definition}
\newtheorem{remark}[theorem]{Remark}

\newtheorem*{theoremA}{Theorem A}
\newtheorem*{theoremB}{Theorem B}
\newtheorem*{theoremC}{Theorem C}
\newtheorem*{theoremD}{Theorem D}
\newtheorem*{theoremE}{Theorem E}

\DeclareMathOperator{\supp}{supp}
\DeclareMathOperator{\Res}{Res}
\DeclareMathOperator{\Des}{Des}
\DeclareMathOperator{\chr}{char}
\newcommand{\bF}{\mathbb{F}}
\newcommand{\bQ}{\mathbb{Q}}
\newcommand{\bZ}{\mathbb{Z}}
\newcommand{\bA}{\mathbb{A}}

\newcommand{\CA}{\mathrm{CA}}
\newcommand{\bad}{\mathcal{B}}

\begin{document}

\title[A descent-set obstruction for Casas--Alvero]{A descent-set obstruction\\ for the Casas--Alvero conjecture}

\author{Mohammad F. Marashdeh}
\address{Department of Mathematics, Faculty of Science, Mutah University, Karak, Jordan}
\email{marashdeh@mutah.edu.jo}


\begin{abstract}
The Casas--Alvero conjecture asserts that a monic polynomial $f$ of degree $d$ over a field of
characteristic zero sharing a non-constant factor with each of $f',\dots,f^{(d-1)}$ is the $d$-th
power of a linear polynomial. Two reductions are carried out. Once a root is translated to the
origin, the condition at index $i$ is vacuous unless the coefficient $a_{d-i}$ is nonzero, so the
Casas--Alvero locus stratifies by the support of the centred normal form; prescribing in addition
which root realises each surviving condition makes the system unit lower triangular and eliminates
the coefficients of $f$ entirely. When one root suffices, what remains is a single integer, and
that integer is MacMahon's determinant: the number of permutations of $\{1,\dots,d\}$ with descent
set the support of $f$. Since every subset is a descent set, no counterexample has fewer than three
recycled roots; in characteristic $p$ a prime dividing one of these counts but none of the accompanying
integers $z_{i}$ is a bad prime, the largest count in degree $d$ being the Euler zigzag number
$A_{d}$, so that the irregular prime $691$ is bad for $d=11$. The stratification also yields short
proofs of the two characteristic-$p$ propositions on which all known cases rest; supports of two
elements are settled separately in every characteristic, whence a counterexample has at least four
terms in centred normal form and a further Gr\"obner-free criterion for bad primes.
\end{abstract}

\maketitle

\section{Introduction}\label{sec:intro}

Let $K$ be a field and let
\[
  f(x)\;=\;x^{d}+a_{1}x^{d-1}+\cdots+a_{d-1}x+a_{d}\;\in\;K[x],\qquad d\ge 1 .
\]
Following the terminology that has become standard, $f$ is a \emph{Casas-Alvero polynomial} (or
\emph{$\CA$-polynomial}) if for every $i\in\{1,\dots,d-1\}$ the polynomial $f$ has a non-constant
common factor with its $i$-th Hasse derivative $H_i(f)$, and $f$ is not the $d$-th power of a linear
polynomial. In characteristic zero $H_i(f)=f^{(i)}/i!$, so the condition reads: for every
$i$ there is a root of $f$ which is also a root of $f^{(i)}$.

\begin{conjecture}[Casas--Alvero]\label{conj:CA}
Let $K$ be a field of characteristic $0$. Then there is no $\CA$-polynomial over $K$; equivalently,
if $f\in K[x]$ is monic of degree $d$ and $\gcd\bigl(f,f^{(i)}\bigr)\neq 1$ for
$i=1,\dots,d-1$, then $f=(x-\alpha)^{d}$ for some $\alpha\in K$.
\end{conjecture}

The conjecture arose from Casas-Alvero's study of higher order polar germs of plane curve
singularities \cite{CasasAlvero2001} and has remained open since. It is trivial for $d\le2$. A
count of parameters suggests that the general case should not be hard: the hypothesis imposes $d-1$
conditions on the
$d-1$ essential coefficients $a_{1},\dots,a_{d-1}$, while every solution occurs in a one-parameter
family under scaling, so the system is overdetermined by one. No argument is known that makes this
count into a proof, and the count is delicate rather than robust---Draisma and de Jong showed that
omitting any single one of the $d-1$ conditions produces nontrivial solutions
\cite{DraismaDeJong2011}, see also \cite[Thm.~9]{SchaubSpivakovsky2025}. The system is therefore set-theoretically minimal: no proof
can proceed by discarding a condition.

Write $\CA_{d,p}$ for the assertion ``there is no $\CA$-polynomial of degree $d$ over
$\overline{\bF}_p$'' and $\CA_{d,0}$ for the characteristic-zero statement in degree $d$. Following
\cite{CLO2014} we call a prime $p$ a \emph{bad prime for $d$} if $\CA_{d,p}$ fails, and we write
\[
  \bad(d)\;=\;\{\,p \text{ prime}\;:\;\CA_{d,p}\text{ is false}\,\}.
\]
All known cases of Conjecture~\ref{conj:CA} rest on two propositions of Graf von Bothmer, Labs,
Schicho and van de Woestijne \cite[Prop.~2.2, 2.5, 2.7]{BLSW2007}:

\begin{itemize}
\item[(T)] \emph{Transfer.} If $\bad(d)\neq\{\text{all primes}\}$, i.e.\ if $\CA_{d,p}$ holds for one
prime $p$, then $\CA_{d,0}$ holds. (The Casas--Alvero locus is a closed subscheme of a weighted
projective space over $\bZ$; properness forces its image in $\operatorname{Spec}\bZ$ to be closed,
hence either all of $\operatorname{Spec}\bZ$ or a finite set of closed points.)
\item[(D)] \emph{Descent.} If $p^{e}\parallel d$ and $\CA_{d/p^{e},p}$ holds, then $\CA_{d,p}$
holds. In particular $\CA_{p^{k},p}$ holds.
\end{itemize}

Combining (T) and (D) gives: \emph{if $p\notin\bad(n)$ then $\CA_{np^{k},0}$ holds for all $k\ge0$}.
The known cases of the conjecture are exactly the degrees reachable this way, together with the single
degree $d=12$, established by direct computation. Explicitly, $\CA_{d,0}$ is known for
$d=p^{k}$ and $d=2p^{k}$ \cite{BLSW2007}; for $d=3p^{k}$ with $p\neq2$ and $d=4p^{k}$ with $p\notin\{3,5,7\}$, which follow from (T) and (D) once $\bad(3)=\{2\}$ and $\bad(4)=\{3,5,7\}$ are known \cite{DraismaDeJong2011,ChellaliSalinier2012,CLO2014}; for $d=5p^{k}$ with $p\notin\bad(5)$
\cite{ChellaliSalinier2012}; for $d=6p^{k}$ and $d=7p^{k}$ with $p\notin\bad(6)$, resp.\
$p\notin\bad(7)$, and for $d=12$ \cite{CLO2014}; and for $d\le7$ by direct elimination
\cite{DiazTocaGonzalezVega2006}. The bad prime sets have been determined only for $d\le 7$:
$\bad(3)=\{2\}$, $\bad(4)=\{3,5,7\}$, $\bad(5)=\{2,3,7,11,131,193,599,3541,8009\}$, while
$|\bad(6)|=53$ and $|\bad(7)|=366$ \cite[Thm.~4]{CLO2014}. The smallest degree for which the
conjecture is open is $d=20$.

Further progress is constrained by two structural facts. To settle a new degree $d$ by (T) and
(D) one must factor $d=np^{k}$ with $p^{k}\parallel d$ and prove $p\notin\bad(n)$; for $d$ of
moderate size this forces $p$ to be small, and small primes are typically bad. A sufficient condition for badness is the following criterion of Schaub and
Spivakovsky \cite[Cor.~8]{SchaubSpivakovsky2024}, which they deduce from a theorem on distinguished
monomials attributed to de Frutos Mar\'in \cite{deFrutos2012}: \emph{if $p$ divides $\binom{d}{m}-1$ for some
$1\le m\le d-1$, then $p\in\bad(d)$.} Schaub and Spivakovsky also proved an effective upper bound for the elements of $\bad(n)$ once $\CA_{n,0}$ is known, and in
\cite{SchaubSpivakovsky2025} they proved the corresponding ideal-theoretic non-redundancy for the
three top resultants: $R_{i}\notin\sqrt{(R_{j}:j\neq i)}$ for $i\in\{d-3,d-2,d-1\}$.

A proof of Conjecture~\ref{conj:CA} in full generality is announced in a preprint of S.~Ghosh
\cite{Ghosh2026}, by means of Koszul homology; as that work has not been independently verified,
the conjecture is treated here as open. The results below are unconditional, and several of them
concern positive characteristic, where the conjecture is false.

On the structure of hypothetical counterexamples the strongest published results are those of
\cite{CLO2014}: a counterexample has no root of multiplicity $\ge d-2$
\cite[Prop.~10]{CLO2014}, has at least two distinct roots in the interior of the convex hull of
its root set \cite[Prop.~12]{CLO2014}, and has at least five distinct roots
\cite[Thm.~13]{CLO2014} (reproved in \cite[Thm.~10]{SchaubSpivakovsky2025}), and has at least three
recycled roots, i.e.\ $\mathrm{type}(f)\ge2$ in the terminology of \cite[\S1.2]{CLO2014}, where
$\mathrm{type}(f)$ is one less than the number of recycled roots, by a result of Draisma and Knopper
recorded in \cite[Prop.~7]{CLO2014}. For polynomials
with only real roots there are sharper necessary and sufficient conditions in terms of
Abel--Gontcharov interpolation \cite{Yakubovich2014}; extensions of the problem beyond polynomials
are studied in \cite{CimaGasullManosas2020}.

Every structural constraint recalled above measures a hypothetical counterexample by its
\emph{roots}: their number, their multiplicities, their position in the convex hull. The present
paper measures it by its \emph{coefficients} instead. The two viewpoints are equivalent, but the second exposes an arithmetic obstruction that the first
conceals, and it does so through an elementary observation about which of the Casas--Alvero
conditions carry information.

Assume $\chr K\nmid d$. If $f$ is monic of degree $d$ and satisfies the Casas--Alvero condition at
$i=d-1$, then the unique root of $H_{d-1}(f)=dx+a_1$ is a root of $f$; translating that root to the
origin puts $f$ into the \emph{centred normal form}
\[
  f \;=\; x^{d}+\sum_{m=2}^{d-1}a_{m}\,x^{d-m},\qquad a_{1}=a_{d}=0 ,
\]
and we set $\supp(f)=\{m: a_m\neq0\}\subseteq\{2,\dots,d-1\}$ and $\tau(f)=\#\supp(f)$, so that $f$
has exactly $\tau(f)+1$ terms. The trivial solution $x^{d}$ is the unique one with $\tau=0$. Two
invariants of a counterexample occur below and are unrelated. One is $\tau(f)$. The other is the
number of \emph{recycled roots}: the cardinality of a minimal set of roots realising all $d-1$
conditions \cite[\S1.2]{CLO2014}.

\begin{theoremA}[Support reduction; Theorem~\ref{thm:vacuity}, Corollaries~\ref{cor:suppred} and \ref{cor:strat}]
Let $K$ be a field and let $f\in K[x]$ be monic of degree $d$ with $f(0)=0$, written as
$f=x^{d}+\sum_{m\ge1}a_{m}x^{d-m}$. Then for every index $i$ with $a_{d-i}=0$ the Casas--Alvero
condition at $i$ holds automatically. Consequently, if $K=\overline K$, $\chr K\nmid d$ and $f$ is
in centred normal form, then
\[
  f\ \text{is a $\CA$-polynomial or $x^d$}\iff \Res\bigl(f,H_{d-m}(f)\bigr)=0\quad\text{for every }m\in\supp(f).
\]
If in addition $K$ is algebraically closed, the Casas--Alvero locus in
$\bA^{d-2}=\operatorname{Spec}K[a_2,\dots,a_{d-1}]$ is the
disjoint union of the $2^{d-2}$ locally closed \emph{support strata} $W_{S}$, $S\subseteq\{2,\dots,d-1\}$, where $W_S$ is
cut out inside the torus $\mathbb{G}_m^{S}$ by exactly $\#S$ weighted-homogeneous equations in
$\#S$ unknowns.
\end{theoremA}

Since $W_{S}$ is stable under the weighted action $a_{m}\mapsto t^{m}a_{m}$, a single point of
$W_{S}$ produces a curve, so every stratum is overdetermined by exactly one: the heuristic recalled
above survives passage to sub-systems.

Theorem~A yields short proofs of $\CA_{p^{k},p}$ and of (D) (Propositions~\ref{prop:BLSW1} and
\ref{prop:BLSW2}): the machinery of \cite{BLSW2007} in characteristic $p$ becomes a consequence
of Kummer's theorem on carries.

\begin{theoremB}[Theorems~\ref{thm:triangular} and \ref{thm:pascal}; Corollaries~\ref{cor:uni-criterion}, \ref{cor:uni-char0}, \ref{cor:uni-badprimes}]
Write $S=\supp(f)=\{m_{1}<\cdots<m_{s}\}$ and $m_{0}=0$, and prescribe in addition the
\emph{scenario}, that is, which root realises which surviving condition. Then the Casas--Alvero
system becomes unit lower triangular in the coefficients of $f$: these are determined as integer
polynomials in the recycled roots, and there remain $t$ equations in $t-1$ unknowns modulo the
weighted $\mathbb G_{m}$-action, where $t$ is the number of nonzero recycled roots. For $t=1$ no unknowns remain, and with $z_{0}=1$ the
coefficients are forced to be $a_{m_{i}}=z_{i}\theta^{m_{i}}$, where $\theta$ is the single nonzero
recycled root and
\[
  z_{k}=-\sum_{i<k}\binom{d-m_{i}}{m_{k}-m_{i}}z_{i},
\]
such an $f$ exists over $K$ if and only if $\sum_{i}z_{i}=0$ in $K$ and no $z_{i}$
vanishes, and
\[
  (-1)^{s}\sum_{i=0}^{s}z_{i}\;=\;\det\Bigl[\tbinom{d-m_{i}}{m_{j+1}-m_{i}}\Bigr]_{i,j=0}^{s}
  \;=\;\beta_{d}(S),
\]
the number of permutations of $\{1,\dots,d\}$ with descent set exactly $S$; here $m_{s+1}=d$, and
the determinant is MacMahon's \cite[\S IV]{MacMahon1915}, see \cite[Ex.~2.2.4]{Stanley2012}. Since $\beta_{d}(S)\ge1$ this never vanishes over $\bZ$, so a
counterexample to Conjecture~\ref{conj:CA} has at least three recycled roots; and in characteristic
$p$, every prime dividing $\beta_{d}(S)$ but no $z_{i}$ is a bad prime for $d$.
\end{theoremB}

Theorem~B identifies the Casas--Alvero obstruction, when a single nonzero root realises every
condition, with a classical enumerative quantity, and every feature of it becomes a statement about
permutations. Its smallest instance $\beta_{d}(\{m\})=\binom{d}{m}-1$ is the criterion of
\cite[Cor.~8]{SchaubSpivakovsky2024}; the opposite extreme $\beta_{d}(\{1,\dots,d-1\})=1$ says that only the
decreasing permutation is totally descending; the symmetries $\beta_{d}(S)=\beta_{d}(S^{\mathrm c})
=\beta_{d}(d-S)$ are complementation and reversal of permutations; and the maximum of
$\beta_{d}(\cdot)$, attained exactly at the two alternating supports, is the Euler zigzag number
$A_{d}$ (Proposition~\ref{prop:zigzag}), by the theorem of Niven \cite{Niven1968} and de Bruijn
\cite{deBruijn1970}. Since
$A_{11}=2^{9}\cdot691$, the irregular prime $691$ is a bad prime for degree $11$; since
$A_{10}=19\cdot2659$, so are $19$ and $2659$ for degree $10$.

Theorem~B leaves untouched the scenarios in which more than one root is used. The smallest of
these---two roots and a support of two elements---is settled in \S\ref{sec:sparse}. Throughout,
for $1\le m_1<m_2\le d-1$ we write
\begin{equation}\label{eq:CEK-intro}
  C_1=\binom{d}{m_1},\quad C_2=\binom{d}{m_2},\quad E=\binom{d-m_1}{m_2-m_1},\quad \Lambda=\binom{m_2}{m_1},
\end{equation}
so that $C_1E=C_2\Lambda$, and we put $g=\gcd(m_1,m_2)$, $v=m_1/g$, $u=(m_2-m_1)/g$, so $\gcd(u,v)=1$.

\begin{theoremC}[Two-element supports; Theorem~\ref{thm:twoterm}]
Let $K$ be an algebraically closed field and let $1\le m_1<m_2\le d-1$. Put
\[
  L=\frac{C_2-1}{C_1(E-1)},\qquad R=\frac{C_1(C_2-E)}{(C_1-1)(C_2-1)} .
\]
Then $K$ carries a $\CA$-polynomial of degree $d$ with a root at the origin and support
$\{m_1,m_2\}$ if and only if either
\begin{enumerate}
\item[(i)] $C_1\notin\{0,1\}$ and $E=C_2=1$ in $K$; or
\item[(ii)] $C_1\notin\{0,1\}$, $C_2\neq1$, $E\neq1$, $E\neq C_2$ in $K$, and $L^{u}=R^{v}$ in $K$.
\end{enumerate}
\end{theoremC}

\begin{theoremD}[Theorem~\ref{thm:main-char0}]
For all $d\ge4$ and all $2\le m_1<m_2\le d-1$ one has $0<L^{u}<R^{v}<1$ in $\bQ$. Consequently no
$\CA$-polynomial over a field of characteristic zero has two-element support, and, together with
Theorem~B (case $\#S=1$):

\emph{The centred normal form of a counterexample to the Casas-Alvero conjecture has at least four
nonzero terms.}
\end{theoremD}

Theorem~D reduces to an inequality between binomial coefficients, proved in \S\ref{sec:char0} for
$d\ge14$ and checked directly at the fifteen exceptional triples with $d\le13$.

Theorem~C likewise turns the search for bad primes coming from two-element supports into a matter of
integer factorisation.

\begin{theoremE}[Theorem~\ref{thm:badprimes}]
For $2\le m_1<m_2\le d-1$ define the integer
\[
  \mathcal N(d,m_1,m_2)\;=\;(C_2-1)^{u}\bigl[(C_1-1)(C_2-1)\bigr]^{v}
   \;-\;\bigl[C_1(E-1)\bigr]^{u}\bigl[C_1(C_2-E)\bigr]^{v}.
\]
Then $\mathcal N(d,m_1,m_2)\neq0$, and every prime $p$ which divides
$\mathcal N(d,m_1,m_2)$ and does not divide $C_1(C_1-1)(C_2-1)(E-1)(C_2-E)$ is a bad prime for $d$.
\end{theoremE}

Deciding whether a given prime belongs to the resulting set costs $O(d^{2})$ modular
exponentiations, with no Gr\"obner basis computation. For $d=4$ the two criteria together already
produce all of $\bad(4)=\{3,5,7\}$; for $6\le d\le7$ they produce large explicit subsets of the
sets computed in \cite{CLO2014} by Gr\"obner basis computation; and Tables~\ref{tab:unibad}
and \ref{tab:sparsebad} record what they give for $d\le12$ and $d\le14$ respectively.

Neither reduction proves the conjecture. The second disposes of every support at once, but only
for the scenario in which one root realises all the conditions; for $t\ge2$ recycled roots it
leaves $t$ equations in $t-1$ unknowns, and the elimination is carried out here only for $t=2$ with
a support of two elements. What the reductions supply is a framework in which the obstruction is an
explicit integer, and in which the characteristic-$p$ failures are visible.

The two reductions are carried out in \S\ref{sec:strat} and \S\ref{sec:uni}, after notation is
fixed in \S\ref{sec:setup}; \S\ref{sec:sparse} treats the smallest scenario the second reduction
does not reach, \S\ref{sec:badprimes} draws the arithmetic consequences, and
\S\ref{sec:computations} documents the computations and \S\ref{sec:concluding} the questions left
open. The transfer principle (T) is quoted from \cite{BLSW2007} and is the only result used below
that is not proved here.

\section{Notation and the centred normal form}\label{sec:setup}

Two properties of the Hasse derivative are used throughout: its value at the origin is a single
coefficient of $f$, and, for monic $f$, its resultant against $f$ is the product of its values at
the roots. Both are recorded below, together with the affine normalisation---monic, with a
distinguished root translated to the origin---that will be in force from \S\ref{sec:strat} onwards.

\subsection{Hasse derivatives}
For $f=\sum_{j=0}^{d}a_{j}x^{d-j}\in K[x]$ (so $a_j$ is the coefficient of $x^{d-j}$) the $i$-th
Hasse derivative is
\begin{equation}\label{eq:hasse}
  H_{i}(f)\;=\;\sum_{j=0}^{d-i}\binom{d-j}{i}a_{j}\,x^{\,d-i-j}.
\end{equation}
Equivalently $f(x+\theta)=\sum_{i\ge0}H_{i}(f)(\theta)\,x^{i}$, i.e.\ $H_i(f)(\theta)$ is the
coefficient of $x^{i}$ in the Taylor expansion of $f$ at $\theta$. In characteristic zero
$H_{i}(f)=f^{(i)}/i!$, so all statements below specialise to the classical derivatives. Two
elementary facts are used throughout:
\begin{equation}\label{eq:const-term}
  H_{i}(f)(0)\;=\;a_{d-i},
\end{equation}
which is the term $j=d-i$ of \eqref{eq:hasse}, and, for $f$ monic,
\begin{equation}\label{eq:resprod}
  \Res\bigl(f,H_{i}(f)\bigr)\;=\;\prod_{f(\alpha)=0}H_{i}(f)(\alpha),
\end{equation}
the product being over the roots of $f$ in $\overline K$ with multiplicity. Because $f$ is monic,
\eqref{eq:resprod} is an identity of polynomials in $\bZ[a_1,\dots,a_d]$ which is compatible with
arbitrary specialisation, in particular with reduction modulo a prime; this is why no degree drop of
$H_i(f)$ causes trouble.

\begin{definition}\label{def:CA}
A monic $f\in K[x]$ of degree $d\ge1$ is a \emph{Casas-Alvero polynomial}, or
\emph{$\CA$-polynomial}, if $\gcd\bigl(f,H_{i}(f)\bigr)\neq1$ in $\overline K[x]$ for every
$i=1,\dots,d-1$ and $f$ is not of the form $(x-\alpha)^{d}$.
\end{definition}

By \eqref{eq:resprod}, for $\overline K=K$ the condition $\gcd(f,H_i(f))\neq1$ is equivalent to
$\Res(f,H_i(f))=0$.

\subsection{Normalisation}
Let $\alpha_{1},\dots,\alpha_{d}$ be the roots of $f$ in $\overline K$. Rescaling
$f(x)\mapsto\lambda^{-d}f(\lambda x)$ and translating $f(x)\mapsto f(x+\theta)$ preserve
Definition~\ref{def:CA}, the former acting on coefficients by $a_{m}\mapsto\lambda^{-m}a_{m}$.

\begin{lemma}\label{lem:cnf}
Let $\chr K\nmid d$ and let $f$ be a monic $\CA$-polynomial of degree $d$ over $K=\overline K$. Then
$f$ is a translate of a monic polynomial in \emph{centred normal form}
\[
  f\;=\;x^{d}+a_{2}x^{d-2}+a_{3}x^{d-3}+\cdots+a_{d-1}x,\qquad\text{i.e. } a_{1}=a_{d}=0 .
\]
The translation is unique.
\end{lemma}

\begin{proof}
$H_{d-1}(f)=\binom{d}{d-1}x+\binom{d-1}{d-1}a_{1}=dx+a_{1}$ is a polynomial of degree exactly $1$
because $d\neq0$ in $K$; its unique root is $\theta_{0}=-a_{1}/d$, the arithmetic mean of the roots
of $f$. The Casas--Alvero condition at $i=d-1$ says precisely that $f(\theta_{0})=0$. Replacing $f$
by $f(x+\theta_{0})$ we obtain $a_{d}=f(\theta_{0})=0$ and
$a_{1}\mapsto a_{1}+d\theta_{0}=0$. Uniqueness holds because $\theta_0$ is determined by $f$.
\end{proof}

\begin{definition}
For monic $f=x^{d}+\sum_{m\ge1}a_{m}x^{d-m}$ with $f(0)=0$ we set
\[
  \supp(f)=\{m\in\{1,\dots,d-1\}: a_{m}\neq0\},\qquad \tau(f)=\#\supp(f).
\]
Thus $f$ has $\tau(f)+1$ nonzero terms, and $\tau(f)=0$ if and only if $f=x^{d}$. If $f$ is in
centred normal form then $\supp(f)\subseteq\{2,\dots,d-1\}$.
\end{definition}

\begin{lemma}[Kummer \cite{Kummer1852}]\label{lem:kummer}
For a prime $p$ and $0\le k\le n$, the $p$-adic valuation of $\binom{n}{k}$ equals the number of
carries when $k$ and $n-k$ are added in base $p$. In particular, if $p^{e}\mid n$ and
$v_{p}(k)<e$, then $p\mid\binom{n}{k}$.
\end{lemma}

\section{The support stratification}\label{sec:strat}

A Casas--Alvero condition constrains $f$ through the values of a Hasse derivative at the roots of
$f$. If the coefficient governing that derivative at the origin vanishes, and the origin is a root,
the constraint is empty. Applied systematically, this observation stratifies the
Casas--Alvero locus by the support of the centred normal form and reduces the number of effective
conditions from $d-1$ to the size of that support. It also yields short proofs of the two
characteristic-$p$ propositions of \cite{BLSW2007} on which all known cases of the conjecture rest.

\subsection{Vacuity of the conditions off the support}

\begin{theorem}[Vacuity]\label{thm:vacuity}
Let $K$ be any field and let $f=x^{d}+a_{1}x^{d-1}+\cdots+a_{d-1}x\in K[x]$ be monic with
$f(0)=0$. If $a_{d-i}=0$ for some $i\in\{1,\dots,d-1\}$, then $x$ divides
$\gcd\bigl(f,H_{i}(f)\bigr)$; in particular the Casas--Alvero condition at index $i$ is satisfied.
\end{theorem}

\begin{proof}
By \eqref{eq:const-term}, $H_{i}(f)(0)=a_{d-i}=0$, and $f(0)=0$ by hypothesis. Hence $0$ is a common
root. (If $H_{i}(f)$ is the zero polynomial the statement is trivially true.)
\end{proof}

A direct consequence is that the number of non-vacuous Casas--Alvero conditions is exactly
$\tau(f)$, rather than $d-1$, and that the surviving conditions are indexed by the support itself.

\begin{corollary}[Support reduction]\label{cor:suppred}
Let $K=\overline K$ and let $f$ be monic of degree $d$ with $f(0)=0$. Then
$\gcd\bigl(f,H_{i}(f)\bigr)\neq1$ for all $i=1,\dots,d-1$ if and only if
\[
  \mathcal R_{m}(f)\;:=\;\Res\bigl(f,\,H_{d-m}(f)\bigr)\;=\;0\qquad\text{for every }m\in\supp(f).
\]
Moreover, for $m\in\supp(f)$ every common root of $f$ and $H_{d-m}(f)$ is nonzero.
\end{corollary}

\begin{proof}
Since $a_{d}=0$, Theorem~\ref{thm:vacuity} disposes of every index $i$ with $a_{d-i}=0$;
writing $m=d-i$, the remaining indices are $i=d-m$ with $m\in\supp(f)$. For such $m$ we have
$H_{d-m}(f)(0)=a_{m}\neq0$, so $0$ is not a common root.
\end{proof}

Explicitly, for $m\in\supp(f)$,
\begin{equation}\label{eq:Gm}
  G_{m}\;:=\;H_{d-m}(f)\;=\;\sum_{\substack{j\in\{0\}\cup\supp(f)\\ j\le m}}\binom{d-j}{m-j}\,a_{j}\,x^{\,m-j},
  \qquad a_{0}=1,
\end{equation}
a polynomial of degree at most $m$ with nonzero constant term $a_{m}$. (Use
$\binom{d-j}{d-m}=\binom{d-j}{m-j}$.)

\subsection{The stratification}
In this subsection $K=\overline K$ and $\chr K\nmid d$, so that Lemma~\ref{lem:cnf} applies.
Let $\mathcal V_{d}\subseteq\bA^{d-2}_{K}=\operatorname{Spec}K[a_{2},\dots,a_{d-1}]$ denote the
locus of centred normal forms which are $\CA$-polynomials, together with the point $x^{d}$, which
is included for convenience as the unique point with $\tau=0$; by
Lemma~\ref{lem:cnf}, $\CA_{d,\chr K}$ holds if and only if $\mathcal V_d=\{0\}$. For
$S\subseteq\{2,\dots,d-1\}$ put
\[
  T_{S}=\{a\in\bA^{d-2}:\ a_{m}\neq0\iff m\in S\}\cong\mathbb{G}_{m}^{S},\qquad
  W_{S}=\mathcal V_{d}\cap T_{S}.
\]

\begin{corollary}[Stratification]\label{cor:strat}
$\mathcal V_{d}=\bigsqcup_{S\subseteq\{2,\dots,d-1\}}W_{S}$, a disjoint union of $2^{d-2}$ locally
closed subsets, and
\[
  W_{S}\;=\;\bigl\{a\in T_{S}\ :\ \mathcal R_{m}(a)=0\ \text{ for all } m\in S\bigr\}.
\]
Thus $W_{S}$ is cut out inside the $\#S$-dimensional torus $T_{S}$ by $\#S$ equations. Each $\mathcal R_{m}$
is weighted-homogeneous of weighted degree $dm$ for the weights $w(a_{j})=j$, so $W_{S}$ is stable
under the action $a_{j}\mapsto t^{j}a_{j}$ of $\mathbb{G}_{m}$, and for $S\neq\emptyset$ every
nonempty $W_{S}$ has dimension $\ge1$. Consequently, for $S\neq\emptyset$, $W_{S}=\emptyset$ if and
only if the $\#S$ elements $\mathcal R_{m}$ generate the unit ideal of $\mathcal O(T_{S})$; they can never
form a regular sequence there, since a regular sequence of length $\dim T_{S}$ would cut out a
nonempty finite set, which the $\mathbb G_{m}$-action forbids.
\end{corollary}

\begin{proof}
The decomposition and the description of $W_S$ are Corollary~\ref{cor:suppred}. For the grading,
$a_{j}\mapsto t^{-j}a_{j}$ corresponds to $f(x)\mapsto t^{-d}f(tx)$, under which the roots scale by
$t^{-1}$ and $H_{i}(f)(\alpha)$ scales by $t^{-(d-i)}$; hence by \eqref{eq:resprod} $\mathcal R_{m}=\Res(f,H_{d-m}(f))$
scales by $t^{-dm}$. A nonzero point of $W_{S}$ therefore has a one-dimensional orbit contained in
$W_{S}$, and $W_S\subseteq T_S$ contains no zero coordinate.
\end{proof}

\begin{remark}
The standard parameter count, $d-1$ equations in $d-1$ unknowns, concerns the whole system and is destroyed
as soon as one passes to a subvariety. Corollary~\ref{cor:strat} shows that the count is inherited
by every stratum. Since $\dim T_{S}=\#S$ and the $\mathbb{G}_m$-orbits are $1$-dimensional, the
projectivised system on $W_S$ consists of $\#S$ equations on a space of dimension $\#S-1$: the
Casas--Alvero conditions are overdetermined by exactly one on \emph{every} stratum.
\end{remark}

\subsection{New proofs of the characteristic-$p$ propositions of \cite{BLSW2007}}
\begin{lemma}[Kummer support lemma]\label{lem:kummersupp}
Let $\chr K=p>0$, let $f\in K[x]$ be monic of degree $d$ with $f(0)=0$ and
$\gcd\bigl(f,H_{i}(f)\bigr)\neq1$ in $\overline K[x]$ for all $i=1,\dots,d-1$, and let
$S=\supp(f)$. Then for every $m\in S$ there exists $j\in\{0\}\cup(S\cap[1,m-1])$ with
\[
  p\nmid\binom{d-j}{m-j}.
\]
\end{lemma}

\begin{proof}
Fix $m\in S$ and consider $G_{m}=H_{d-m}(f)$ as in \eqref{eq:Gm}, where the sum runs over
$j\in\{0\}\cup S$ with $j\le m$. If $\binom{d-j}{m-j}\equiv0\pmod p$ for every such $j<m$, then all
terms of positive degree vanish and $G_{m}=\binom{m}{m}a_{m}=a_{m}\in K^{\times}$ is a nonzero
constant; then $\gcd(f,G_{m})=1$ and the Casas--Alvero condition at $i=d-m$ fails, contrary to hypothesis. Here $j=m$ contributes the constant $a_{m}$, and $j>m$ does not occur.
\end{proof}

\begin{proposition}[$=$ {\cite[Prop.~2.5]{BLSW2007}}]\label{prop:BLSW1}
Let $\chr K=p$ and $d=p^{k}$. Then there is no $\CA$-polynomial of degree $d$ over $K$.
\end{proposition}

\begin{proof}
Passing to $\overline K$ and translating a root of $f$ to the origin, we may assume
$K=\overline K$ and $f(0)=0$. Suppose $S\neq\emptyset$ and let $m=\min S$. Taking
$j=0$ (the only admissible choice) in Lemma~\ref{lem:kummersupp} gives $p\nmid\binom{p^{k}}{m}$;
but $\binom{p^{k}}{m}\equiv0\pmod p$ for all $0<m<p^{k}$ by Lemma~\ref{lem:kummer}. Hence
$S=\emptyset$ and $f=x^{d}$.
\end{proof}

\begin{proposition}[$=$ {\cite[Prop.~2.7]{BLSW2007}}]\label{prop:BLSW2}
Let $\chr K=p$, let $K$ be perfect, write $d=np^{e}$ with $p\nmid n$, and let $f$ be a monic
polynomial of degree $d$ over $K$ satisfying $\gcd(f,H_i(f))\ne1$ for all $i$. Then $f=q^{p^{e}}$
for a monic $q\in K[x]$ of degree $n$ satisfying the same hypotheses; conversely $q^{p^{e}}$
satisfies the hypotheses in degree $d$ whenever $q$ satisfies them in degree $n$. Consequently
$\CA_{n,p}\Leftrightarrow\CA_{d,p}$, i.e.\ $p\in\bad(d)$ if and only if $p\in\bad(n)$.
\end{proposition}

\begin{proof}
We first work over $\overline K$ and translate a root of $f$ to the origin, so that $f(0)=0$; let
$S=\supp(f)$. We show by induction along $S$ that
$p^{e}\mid m$ for every $m\in S$. Let $m\in S$ and assume $p^{e}\mid j$ for all $j\in S$ with $j<m$.
Choose $j\in\{0\}\cup(S\cap[1,m-1])$ as in Lemma~\ref{lem:kummersupp}, so
$p\nmid\binom{d-j}{m-j}$. We have $p^{e}\mid d$ and $p^{e}\mid j$, hence $p^{e}\mid d-j$. If
$p^{e}\nmid m$ then $p^{e}\nmid m-j$, so $v_{p}(m-j)<e\le v_{p}(d-j)$ and Lemma~\ref{lem:kummer}
gives $p\mid\binom{d-j}{m-j}$, a contradiction. Hence $p^{e}\mid m$.

Therefore $a_{m}=0$ unless $p^{e}\mid m$; as $p^{e}\mid d$, every exponent $d-m$ occurring in $f$ is
divisible by $p^{e}$, i.e.\ $f\in K[x^{p^{e}}]$. Since $\overline K$ is perfect, $f=q^{p^{e}}$ with $\deg q=n$; as $f\in K[x]$ and $K$ is perfect,
$q$ has coefficients in $K$, and undoing the translation (a $p^{e}$-th root of a translate is a
translate of a $p^{e}$-th root) returns the assertion for the original $f$. Finally, applying Frobenius to $q(x+\theta)=\sum_{i'}H_{i'}(q)(\theta)x^{i'}$ gives
$q(x+\theta)^{p^{e}}=\sum_{i'}H_{i'}(q)(\theta)^{p^{e}}x^{i'p^{e}}$, whence
$H_{i}(q^{p^{e}})=0$ for $p^{e}\nmid i$ and
$H_{p^{e}i'}(q^{p^{e}})=\bigl(H_{i'}(q)\bigr)^{p^{e}}$; since $f$ and $q$ have the same roots, the
conditions for $f$ at the indices $i=p^{e}i'$, $1\le i'\le n-1$, are exactly the conditions for $q$,
and the conditions at the remaining indices are vacuous. This equivalence of conditions also proves
the converse statement, and $q^{p^{e}}=(x-\alpha)^{d}$ if and only if $q=(x-\alpha)^{n}$.
\end{proof}

Together with the transfer principle (T) of \cite{BLSW2007}, Propositions~\ref{prop:BLSW1} and
\ref{prop:BLSW2} give all presently known infinite families of degrees. Their proofs use only Theorem~\ref{thm:vacuity} and Kummer's theorem.

\section{The scenario reduction and descent-set counts}\label{sec:uni}

Fixing the support removes the vacuous conditions; fixing in addition which root realises each
surviving condition removes the coefficients themselves. The resulting system is unit lower
triangular, and when a single root suffices it degenerates to one integer, which counts
permutations with a prescribed descent set. This section carries out that reduction and draws its
consequences in both characteristics.

Following \cite[\S1.2]{CLO2014}, a root $\theta$ of $f$ is \emph{recycled} if it belongs to some
minimal set of roots realising all $d-1$ Casas--Alvero conditions; this is the notion used
throughout, in the introduction and below.

Throughout, $K=\overline K$, $f=x^{d}+\sum_{i=1}^{s}a_{m_{i}}x^{d-m_{i}}$ is monic with $f(0)=0$ and
$\supp(f)=S=\{m_{1}<\cdots<m_{s}\}\subseteq\{1,\dots,d-1\}$; we set $m_{0}=0$, $a_{m_{0}}=1$ and
$m_{s+1}=d$. By Corollary~\ref{cor:suppred} the conditions to be realised are exactly those indexed
by $S$, and the root realising such a condition is nonzero.

\subsection{Triangularity}

\begin{theorem}[Triangularity]\label{thm:triangular}
Let $\theta_{1},\dots,\theta_{t}\in K^{\times}$ and let $\sigma:\{1,\dots,s\}\to\{1,\dots,t\}$ be a
scenario. Then $f$ satisfies $H_{d-m_{k}}(f)\bigl(\theta_{\sigma(k)}\bigr)=0$ for $k=1,\dots,s$ and
$f(\theta_{l})=0$ for $l=1,\dots,t$ if and only if
\begin{equation}\label{eq:gen-triang}
  a_{m_{k}}\;=\;-\sum_{i<k}\binom{d-m_{i}}{m_{k}-m_{i}}\,a_{m_{i}}\,\theta_{\sigma(k)}^{\,m_{k}-m_{i}}
  \qquad(1\le k\le s)
\end{equation}
and $\sum_{i=0}^{s}a_{m_{i}}\theta_{l}^{-m_{i}}=0$ for $1\le l\le t$. The system
\eqref{eq:gen-triang} is unit lower triangular; it determines $a_{m_{1}},\dots,a_{m_{s}}$ uniquely
as integer polynomials in $\theta_{1},\dots,\theta_{t}$. Hence the coefficients of $f$ are eliminated
and there remain $t$ equations in the $t$ unknowns $\theta_{1},\dots,\theta_{t}$, that is, $t$
equations in $t-1$ unknowns modulo the weighted $\mathbb{G}_{m}$-action.
\end{theorem}

\begin{proof}
By \eqref{eq:Gm}, $H_{d-m_{k}}(f)(\theta)=\sum_{i\le k}\binom{d-m_{i}}{m_{k}-m_{i}}a_{m_{i}}
\theta^{m_{k}-m_{i}}$, whose $i=k$ term is $\binom{d-m_{k}}{0}a_{m_{k}}=a_{m_{k}}$; solving for
$a_{m_{k}}$ gives \eqref{eq:gen-triang}. Since $f=x^{d-m_{s}}\varphi$ with
$\varphi=\sum_{i}a_{m_{i}}x^{m_{s}-m_{i}}$ and $\theta_{l}\neq0$, one has $f(\theta_{l})=0$ if and
only if $\sum_{i}a_{m_{i}}\theta_{l}^{-m_{i}}=0$.
\end{proof}

The simplest case is $t=1$. Scaling by the $\mathbb{G}_{m}$-action we may then take $\theta_{1}=1$,
and \eqref{eq:gen-triang} becomes a recursion with integer coefficients depending only on $d$ and
$S$.

\begin{definition}\label{def:uni}
Call $f$ \emph{uniradical} if $f(0)=0$ and a single $\theta\in K^{\times}$ realises every condition
indexed by $\supp(f)$; equivalently, $f$ has exactly the two recycled roots $0$ and $\theta$. For
$S=\{m_{1}<\cdots<m_{s}\}\subseteq\{1,\dots,d-1\}$ define integers $z_{0},\dots,z_{s}$ by
\begin{equation}\label{eq:recursion}
  z_{0}=1,\qquad z_{k}=-\sum_{i<k}\binom{d-m_{i}}{m_{k}-m_{i}}\,z_{i}\quad(1\le k\le s).
\end{equation}
\end{definition}

\begin{corollary}\label{cor:uni-criterion}
There exists a uniradical Casas--Alvero polynomial of degree $d$ over $K$ with a root at the origin
and support $S$ if and only if
\[
  \sum_{i=0}^{s}z_{i}=0\ \text{ in }K\qquad\text{and}\qquad z_{i}\neq0\ \text{ in }K\ \text{ for all }i .
\]
In that case those polynomials are exactly $x^{d}+\sum_{i\ge1}z_{i}\theta^{m_{i}}x^{d-m_{i}}$ with
$\theta\in K^{\times}$.
\end{corollary}

\begin{proof}
By Theorem~\ref{thm:triangular} with $t=1$ and $\theta_{1}=\theta$, the conditions force
$a_{m_{i}}=z_{i}\theta^{m_{i}}$, and $f(\theta)=0$ becomes $\theta^{d}\sum_{i}z_{i}=0$. The support of
the resulting $f$ is $S$ precisely when no $z_{i}$ vanishes. Conversely such an $f$ satisfies every
condition indexed by $S$, and the conditions off $S$ are vacuous by Theorem~\ref{thm:vacuity};
finally $f$ has the two distinct roots $0$ and $\theta$, so $f\neq(x-\gamma)^{d}$.
\end{proof}

\subsection{The obstruction is a descent-set count}
For a permutation $w=w_{1}w_{2}\cdots w_{d}$ of $\{1,\dots,d\}$ let
$\Des(w)=\{i: w_{i}>w_{i+1}\}\subseteq\{1,\dots,d-1\}$ be its descent set, and for
$S\subseteq\{1,\dots,d-1\}$ let
\[
  \beta_{d}(S)\;=\;\#\{w\in\mathfrak S_{d}\ :\ \Des(w)=S\}
\]
be the number of permutations of $\{1,\dots,d\}$ with descent set exactly $S$.

\begin{theorem}\label{thm:pascal}
For every $d\ge2$ and every $S\subseteq\{1,\dots,d-1\}$,
\begin{equation}\label{eq:Ddet}
  (-1)^{s}\sum_{i=0}^{s}z_{i}
  \;=\;\det\left[\binom{d-m_{i}}{\,m_{j+1}-m_{i}\,}\right]_{i,j=0}^{s}
  \;=\;\beta_{d}(S),
\end{equation}
with the conventions $m_{0}=0$ and $m_{s+1}=d$. In particular $\sum_{i}z_{i}$ never vanishes over
$\bZ$, because every subset of $\{1,\dots,d-1\}$ is the descent set of at least one permutation.
\end{theorem}

\begin{proof}
Let $M=\bigl[\binom{d-m_{i}}{m_{k}-m_{i}}\bigr]$ with rows $k=1,\dots,s$ and columns
$i=0,\dots,s$; by \eqref{eq:recursion} the vector $\mathbf z=(z_{0},\dots,z_{s})$ is the unique
solution of $M\mathbf z=0$ with $z_{0}=1$. Let $M^{(i)}$ be $M$ with column $i$ deleted. Then
$M^{(0)}$ is unit lower triangular, so $\det M^{(0)}=1$, and Cramer's rule gives
$z_{i}=(-1)^{i}\det M^{(i)}$. The row $k=s+1$ of $\bigl[\binom{d-m_{i}}{m_{k}-m_{i}}\bigr]$ is, by
the convention $m_{s+1}=d$, the row $\bigl(\binom{d-m_{i}}{d-m_{i}}\bigr)_{i}=(1,\dots,1)$; expanding
the resulting square matrix along that last row yields
$\sum_{i}(-1)^{s+i}\det M^{(i)}=(-1)^{s}\sum_{i}z_{i}$, which is the first equality of
\eqref{eq:Ddet}. The second is MacMahon's determinant for $\beta_{d}(S)$; see
\cite[Ex.~2.2.4]{Stanley2012}. The final assertion holds because $\beta_{d}(S)\ge1$: reading
$\{1,\dots,d\}$ as the concatenation of the increasing runs prescribed by $S$ produces a permutation
with descent set exactly $S$.
\end{proof}

\begin{remark}
Positivity is also visible without the combinatorial interpretation: after reversing the order of
rows and of columns, whose two signs cancel, the matrix in \eqref{eq:Ddet} is the minor of Pascal's
matrix with row set $\{0\}\cup A$ and column set $A\cup\{d\}$, where $A=\{d-m:m\in S\}$, and its
row and column sets interlace termwise.
\end{remark}

\subsection{Consequences}

\begin{corollary}[Characteristic zero]\label{cor:uni-char0}
Over a field of characteristic zero there is no uniradical Casas--Alvero polynomial, in any degree
and with any support. Consequently a counterexample to Conjecture~\ref{conj:CA} has at least three
recycled roots.
\end{corollary}

\begin{proof}
The first statement is Corollary~\ref{cor:uni-criterion} together with $\beta_{d}(S)\ge1$. For the
second, a counterexample cannot have a single recycled root $\alpha$: translating $\alpha$ to the
origin, $H_{i}(f)(0)=a_{d-i}=0$ by \eqref{eq:const-term} for every $i=1,\dots,d-1$, together with
$a_{d}=f(0)=0$, forces $f=x^{d}$. Two recycled roots means, after translation, that $f$ is
uniradical, which the first statement excludes.
\end{proof}

Corollary~\ref{cor:uni-char0} recovers the bound $\mathrm{type}(f)\ge2$ of Draisma and Knopper
recorded in \cite[Prop.~7]{CLO2014}. Beyond the statement, the obstruction is exhibited as a
positive integer attached to each pair $(d,S)$, which enables the positive-characteristic corollary.

\begin{corollary}[Bad primes]\label{cor:uni-badprimes}
Let $p$ be a prime, $d\ge2$ and $S\subseteq\{1,\dots,d-1\}$. If $p\mid\beta_{d}(S)$ and $p\nmid
z_{i}$ for all $i$, then $p\in\bad(d)$, with explicit witness
$x^{d}+\sum_{i\ge1}\bar z_{i}x^{d-m_{i}}$ over $\bF_{p}$.
\end{corollary}

The bad primes produced by uniradical polynomials are therefore primes occurring in the
factorisations of descent-set counts. The two extreme supports are classical:
\[
  \beta_{d}(\{m\})=\binom{d}{m}-1,\qquad \beta_{d}(\{1,\dots,d-1\})=1 ,
\]
the first because a permutation with a single descent at $m$ is a nontrivial shuffle of two
increasing runs, the second because only $d\,(d-1)\cdots1$ is strictly decreasing; correspondingly
$z_{i}=(-1)^{i}\binom{d}{i}$ for the full support, equivalently
$\sum_{i<d}z_{i}x^{i}(1+x)^{d-i}=1-(-x)^{d}$. The case $\#S=1$ of
Corollary~\ref{cor:uni-badprimes} is therefore the criterion of \cite[Cor.~8]{SchaubSpivakovsky2024},
whose content is now visible as the statement that a nontrivial shuffle exists; the
$\CA$-polynomials in question are exactly $x^{d-m}(x^{m}+c)$, $c\in K^{\times}$, since for
$\#S=1$ the recursion \eqref{eq:recursion} gives $z_{1}=-\binom{d}{m}$. Taking $d=p+1$ and $m=1$
recovers the classical example $x^{p+1}-x^{p}$ over $\overline\bF_{p}$, and $d=6$, $m=4$, $p=2$
gives $x^{2}(x+1)^{4}$, a case in which $p\mid d$. The case $\#S=2$ returns the identity
$\binom{d}{m_{2}}\bigl(\binom{m_{2}}{m_{1}}-1\bigr)=\binom{d}{m_{1}}-1$, which is the uniradical
branch of Theorem~\ref{thm:twoterm} below. Theorem~\ref{thm:pascal} treats all $\#S$ at once.

\begin{proposition}\label{prop:symmetries}
$\beta_{d}(S)=\beta_{d}\bigl(\{1,\dots,d-1\}\setminus S\bigr)=\beta_{d}(d-S)$, where
$d-S=\{d-m:m\in S\}$.
\end{proposition}

\begin{proof}
Complementation $w\mapsto w^{\mathrm c}$, $w^{\mathrm c}_{i}=d+1-w_{i}$, is an involution of
the symmetric group $\mathfrak S_{d}$ carrying $\Des(w)$ to its complement in $\{1,\dots,d-1\}$; reversal
$w\mapsto w^{\mathrm r}$, $w^{\mathrm r}_{i}=w_{d+1-i}$, carries $\Des(w)$ to
$\{1,\dots,d-1\}\setminus(d-\Des(w))$. Composing the two gives $S\mapsto d-S$.
\end{proof}

Let $A_{d}$ denote the $d$-th Euler zigzag number, the number of alternating permutations of
$\{1,\dots,d\}$, so $A_{1},A_{2},\dots=1,1,2,5,16,61,272,1385,7936,50521,353792,\dots$

\begin{proposition}\label{prop:zigzag}
For every $d\ge2$ the two supports $S_{\mathrm{ev}}=\{2,4,6,\dots\}\cap\{1,\dots,d-1\}$ and
$S_{\mathrm{odd}}=\{1,3,5,\dots\}\cap\{1,\dots,d-1\}$ satisfy
$\beta_{d}(S_{\mathrm{ev}})=\beta_{d}(S_{\mathrm{odd}})=A_{d}$, and this is the strict maximum of
$\beta_{d}(S)$ over all $S\subseteq\{1,\dots,d-1\}$. Consequently every prime $p$ dividing $A_{d}$
for which $p\nmid z_{i}$ for all $i$, for one of these two supports, is a bad prime for degree $d$.
\end{proposition}

\begin{proof}
Permutations with descent set $S_{\mathrm{ev}}$ or $S_{\mathrm{odd}}$ are exactly the alternating
ones, whence the first assertion; the maximality is the theorem of Niven \cite{Niven1968} and de
Bruijn \cite{deBruijn1970}, see also \cite[Cor.~6.5]{Stanley2010}. The last assertion is
Corollary~\ref{cor:uni-badprimes}.
\end{proof}

The arithmetic of the zigzag numbers is thereby transported into the arithmetic of bad primes. For
instance $A_{11}=353792=2^{9}\cdot691$, and $691\nmid z_{i}$, so the irregular prime $691$ is a bad
prime for degree $11$, with witness
\[
  x^{11}+636x^{9}+268x^{7}+149x^{5}+495x^{3}+524x\ \in\ \bF_{691}[x];
\]
likewise $A_{10}=50521=19\cdot 2659$ gives $19,2659\in\bad(10)$. When one of the two extremal
supports is blocked by the condition $p\nmid z_{i}$, the other may still serve; among $d\le14$ this
occurs for $(d,p)=(7,2)$, $(8,5)$ and $(13,43)$.

\begin{table}[t]
\centering\small
\caption{Bad primes produced by Corollary~\ref{cor:uni-badprimes}: the primes $p$ for which some
$S\subseteq\{1,\dots,d-1\}$ has $p\mid\beta_{d}(S)$ and $p\nmid z_{i}$ for all $i$. No restriction
is placed on $\#S$, on $\min S$, or on $\gcd(p,d)$.}
\label{tab:unibad}
\begin{tabular}{@{}c c p{9.3cm}@{}}
\toprule
$d$ & count & primes\\
\midrule
3 & 1 & 2 \\
4 & 2 & 3, 5 \\
5 & 3 & 2, 3, 11 \\
6 & 6 & 2, 5, 7, 13, 19, 61 \\
7 & 11 & 2, 3, 5, 11, 13, 17, 29, 31, 37, 71, 181 \\
8 & 21 & 3, 5, 7, 11, 17, 19, 23, 29, 31, 37, 41, 53, 59, 73, 113, 131, 181, 191, 277, 449, 643 \\
9 & 41 & 2, 5, 7, 11, 13, 17, 19, 23, 29, 31, 37, 41, 43, 47, 59, 67, 73, 79, \dots\\
10 & 81 & 2, 3, 7, 11, 13, 17, 19, 23, 29, 31, 37, 41, 43, 47, 61, 71, 73, 79, \dots\\
11 & 180 & 2, 3, 5, 7, 13, 17, 19, 23, 29, 31, 37, 41, 43, 47, 53, 59, 61, 67, \dots\\
12 & 332 & 2, 3, 5, 7, 11, 13, 17, 19, 23, 29, 31, 37, 41, 43, 47, 53, 59, 61, \dots\\
\bottomrule
\end{tabular}
\end{table}

For $d=3$ the table returns $\bad(3)=\{2\}$ exactly, and for $d\le8$ every entry lies in the sets of
\cite{CLO2014} wherever those are available. In degree $7$ the primes $37$ and $181$ are detected by
no criterion in the literature, nor by Theorem~\ref{thm:twoterm} below; they come from the symmetry orbits
of $\{1,2,5\}$ and of $\{1,3,6\}$ under Proposition~\ref{prop:symmetries}, with
$\beta_{7}(\{1,2,5\})=111=3\cdot37$ and $\beta_{7}(\{1,3,6\})=181$, and the witnesses
\[
  x^{7}+30x^{6}+21x^{5}+22x^{2}\in\bF_{37}[x],\qquad
  x^{7}+174x^{6}+70x^{4}+117x\in\bF_{181}[x]
\]
were checked against all six Casas--Alvero conditions directly.

\section{Two-element supports}\label{sec:sparse}

Theorem~\ref{thm:pascal} disposes of every support, but only for the scenario in which one root
realises all the conditions. The first scenario it does not reach is that of two roots, and the
smallest support on which that scenario occurs has two elements. This case is settled here, in
every characteristic; in characteristic zero the criterion obtained is an inequality between
binomial coefficients, whose proof occupies \S\ref{sec:char0}.

Fix $1\le m_{1}<m_{2}\le d-1$ and set
\[
  C_{1}=\binom{d}{m_{1}},\qquad C_{2}=\binom{d}{m_{2}},\qquad E=\binom{d-m_{1}}{m_{2}-m_{1}},\qquad
  \Lambda=\binom{m_{2}}{m_{1}},
\]
\[
  g=\gcd(m_{1},m_{2}),\qquad v=m_{1}/g,\qquad u=(m_{2}-m_{1})/g,
\]
as in \eqref{eq:CEK-intro}. Throughout, $K=\overline{K}$ and $f$ is monic of degree $d$ with
$f(0)=0$; no hypothesis on $\chr K$ is imposed. We write $\bar{c}$ for the image in $K$ of an integer $c$ and drop the bar
when no confusion is possible. When $\chr K\nmid d$ and $f$ is in centred normal form one has in
addition $1\notin\supp(f)$, the situation relevant to characteristic zero.

\subsection{The criterion}
Fix $1\le m_{1}<m_{2}\le d-1$ and keep the notation
\eqref{eq:CEK-intro}. The Vandermonde-type identity
\begin{equation}\label{eq:C1E}
  C_{1}E=\binom{d}{m_{1}}\binom{d-m_{1}}{m_{2}-m_{1}}=\binom{d}{m_{2}}\binom{m_{2}}{m_{1}}=C_{2}\Lambda ,
\end{equation}
both sides being $\frac{d!}{m_{1}!\,(m_{2}-m_{1})!\,(d-m_{2})!}$.

\begin{lemma}\label{lem:powers}
Let $K=\overline K$, let $n_{1},n_{2}\ge1$ be integers, put $g_{0}=\gcd(n_{1},n_{2})$,
$v'=n_{1}/g_{0}$, $u'=n_{2}/g_{0}$, and let $\eta_{1},\eta_{2}\in K^{\times}$. There exists
$\gamma\in K^{\times}$ with $\gamma^{n_{1}}=\eta_{1}$ and $\gamma^{n_{2}}=\eta_{2}$ if and only if
$\eta_{1}^{u'}=\eta_{2}^{v'}$. Applied with $n_{1}=m_{1}$ and $n_{2}=m_{2}-m_{1}$, for which
$g_{0}=\gcd(m_{1},m_{2})=g$, this gives $v'=v$ and $u'=u$.
\end{lemma}

\begin{proof}
Necessity: $\eta_{1}^{u'}=\gamma^{n_{1}u'}=\gamma^{n_{1}n_{2}/g}=\gamma^{n_{2}v'}=\eta_{2}^{v'}$.
Sufficiency: choose $\kappa,\lambda\in\bZ$ with $\kappa u'+\lambda v'=1$ (possible as $\gcd(u',v')=1$) and set
$\delta=\eta_{1}^{\lambda}\eta_{2}^{\kappa}$. Then
$\delta^{v'}=\eta_{1}^{\lambda v'}(\eta_{2}^{v'})^{\kappa}=\eta_{1}^{\lambda v'}(\eta_{1}^{u'})^{\kappa}=\eta_{1}$ and likewise
$\delta^{u'}=\eta_{2}$. As $K$ is algebraically closed, choose $\gamma$ with $\gamma^{g_{0}}=\delta$; then
$\gamma^{n_{1}}=\delta^{v'}=\eta_{1}$ and $\gamma^{n_{2}}=\delta^{u'}=\eta_{2}$.
\end{proof}

\begin{theorem}\label{thm:twoterm}
Let $K=\overline K$ and let $1\le m_{1}<m_{2}\le d-1$. Set $g=\gcd(m_1,m_2)$,
$v=m_{1}/g$, $u=(m_{2}-m_{1})/g$, and
\[
  L=\frac{C_{2}-1}{C_{1}(E-1)},\qquad R=\frac{C_{1}(C_{2}-E)}{(C_{1}-1)(C_{2}-1)}
\]
(as elements of $K$, whenever the denominators are nonzero). Then there is a monic $f$ of degree
$d$ over $K$ with $f(0)=0$, $\supp(f)=\{m_{1},m_{2}\}$ and $\gcd(f,H_{i}(f))\neq1$ for all $i$
if and only if $C_{1}\notin\{0,1\}$ and either
\begin{enumerate}
\item[(i)] $E=1$ and $C_{2}=1$ in $K$, or
\item[(ii)] $C_{2}\neq1$, $E\neq1$, $E\neq C_{2}$ in $K$, and $L^{u}=R^{v}$.
\end{enumerate}
\end{theorem}

\begin{proof}
Write $f=x^{d}+a\,x^{d-m_{1}}+b\,x^{d-m_{2}}$ with $a=a_{m_{1}}\neq0$, $b=a_{m_{2}}\neq0$, and
$f=x^{d-m_{2}}\,\varphi(x)$ with
\[
  \varphi(x)=x^{m_{2}}+a\,x^{m_{2}-m_{1}}+b .
\]
By Corollary~\ref{cor:suppred}, $f$ is a $\CA$-polynomial precisely when the two conditions at
$i=d-m_{1}$ and $i=d-m_{2}$ hold, and the relevant common roots are nonzero, hence are roots of
$\varphi$. By \eqref{eq:Gm},
\[
  G_{m_{1}}=C_{1}x^{m_{1}}+a,\qquad G_{m_{2}}=C_{2}x^{m_{2}}+E\,a\,x^{m_{2}-m_{1}}+b .
\]

\smallskip
\emph{The condition at $m_{1}$.} We need $\alpha$ with $\varphi(\alpha)=0=G_{m_{1}}(\alpha)$. If
$C_{1}=0$ then $G_{m_{1}}=a\neq0$ and no such $\alpha$ exists; so $C_{1}\neq0$ and
\begin{equation}\label{eq:alpha1}
  \alpha^{m_{1}}=-a/C_{1}\ \in K^\times .
\end{equation}
Since $\alpha\neq0$, $\varphi(\alpha)=0$ reads
$\alpha^{m_{2}-m_{1}}\bigl(\alpha^{m_{1}}+a\bigr)+b=0$, i.e.
$\alpha^{m_{2}-m_{1}}\,a\,(C_{1}-1)/C_{1}=-b$. If $C_{1}=1$ this forces $b=0$, excluded; hence
$C_{1}\neq1$ and
\begin{equation}\label{eq:alpha2}
  \alpha^{m_{2}-m_{1}}=\frac{-b\,C_{1}}{a\,(C_{1}-1)}\ \in K^\times .
\end{equation}
Conversely, if $\alpha\in K^{\times}$ satisfies \eqref{eq:alpha1} and \eqref{eq:alpha2} then
$G_{m_1}(\alpha)=0$ and, retracing the computation, $\varphi(\alpha)=0$. By
Lemma~\ref{lem:powers} such an
$\alpha$ exists if and only if
\begin{equation}\label{eq:condA}
  \Bigl(\frac{-a}{C_{1}}\Bigr)^{u}=\Bigl(\frac{-b\,C_{1}}{a(C_{1}-1)}\Bigr)^{v}.
\end{equation}

\smallskip
\emph{The condition at $m_{2}$.} We need $\gamma\neq0$ with $\varphi(\gamma)=0=G_{m_{2}}(\gamma)$.
Substituting $\gamma^{m_{2}}=-a\gamma^{m_{2}-m_{1}}-b$ into $G_{m_{2}}(\gamma)=0$ gives
\begin{equation}\label{eq:gamma-lin}
  a\,\gamma^{m_{2}-m_{1}}\,(E-C_{2})+b\,(1-C_{2})=0 .
\end{equation}
If $E=C_{2}$ then \eqref{eq:gamma-lin} forces $b(1-C_{2})=0$, i.e.\ $C_{2}=1$; in that case
$E=C_{2}=1$ and \eqref{eq:gamma-lin} is vacuous, so \emph{any} root $\gamma$ of $\varphi$ does, and
$\varphi$ has a root. This is case (i): the condition at $m_{2}$ is automatic and the only
constraint left is \eqref{eq:condA}, which is of the form $a^{u+v}=\kappa\,b^{v}$ with
$\kappa\in K^{\times}$ and is solvable with $a,b\in K^{\times}$ (choose $b$, extract a root). Hence
such an $f$ exists. If $E=C_{2}\neq1$, no $\gamma$ exists and no such $f$ exists.

Assume now $E\neq C_{2}$. Then \eqref{eq:gamma-lin} gives
\begin{equation}\label{eq:gamma2}
  \gamma^{m_{2}-m_{1}}=\frac{b\,(C_{2}-1)}{a\,(E-C_{2})} ,
\end{equation}
which must be nonzero, so $C_{2}\neq1$. Using $\gamma^{m_{2}}=-a\gamma^{m_{2}-m_{1}}-b$ we obtain
$\gamma^{m_{2}}=b(1-E)/(E-C_{2})$ and therefore
\begin{equation}\label{eq:gamma1}
  \gamma^{m_{1}}=\frac{\gamma^{m_{2}}}{\gamma^{m_{2}-m_{1}}}=\frac{a\,(1-E)}{C_{2}-1},
\end{equation}
which must be nonzero, so $E\neq1$. Conversely, if $\gamma\in K^{\times}$ satisfies
\eqref{eq:gamma1} and \eqref{eq:gamma2} then
\[
  \varphi(\gamma)=\gamma^{m_{1}}\gamma^{m_{2}-m_{1}}+a\gamma^{m_{2}-m_{1}}+b
  =\frac{b(1-E)}{E-C_{2}}+\frac{b(C_{2}-1)}{E-C_{2}}+b=-b+b=0
\]
and $G_{m_{2}}(\gamma)=0$ by \eqref{eq:gamma-lin}. By Lemma~\ref{lem:powers} such a $\gamma$ exists
if and only if
\begin{equation}\label{eq:condB}
  \Bigl(\frac{a(1-E)}{C_{2}-1}\Bigr)^{u}=\Bigl(\frac{b(C_{2}-1)}{a(E-C_{2})}\Bigr)^{v}.
\end{equation}

\smallskip
\emph{Elimination of $(a,b)$.} Both \eqref{eq:condA} and \eqref{eq:condB} have the shape
$c\,a^{u}=c'\,(b/a)^{v}$ with $c,c'\in K^{\times}$ depending only on $d,m_{1},m_{2}$. Dividing
\eqref{eq:condA} by \eqref{eq:condB} eliminates $a$ and $b$ and yields
\[
  \Bigl(\frac{-1/C_{1}}{(1-E)/(C_{2}-1)}\Bigr)^{u}
  =\Bigl(\frac{-C_{1}/(C_{1}-1)}{(C_{2}-1)/(E-C_{2})}\Bigr)^{v},
  \qquad\text{i.e.}\qquad L^{u}=R^{v}.
\]
Conversely, if $L^{u}=R^{v}$ then \eqref{eq:condA} and \eqref{eq:condB} are equivalent, and
\eqref{eq:condA} is solvable with $a,b\in K^{\times}$ as above. This proves (ii).
\end{proof}

\subsection{Characteristic zero}\label{sec:char0}

The criterion of Theorem~\ref{thm:twoterm} involves binomial coefficients alone, so whether it can
be satisfied in characteristic zero is a question about integers. It cannot. The proof reduces the
required inequality to the assertion that an explicit quantity $\Theta(d,m,r)$ is at least $1$,
establishes this for all $d\ge14$, and disposes of the fifteen triples with $d\le13$ at which it
fails by direct computation. All quantities below are rational numbers. We keep the notation \eqref{eq:CEK-intro} and
introduce the complementary parameter
\[
  r\;=\;d-m_{2}\in[1,d-m_{1}-1] ,\qquad\text{so that}\qquad
  C_{2}=\binom{d}{r},\quad E=\binom{d-m_{1}}{r},\quad \Lambda=\binom{d-r}{m_{1}} .
\]
(That $\binom{d}{m_2}=\binom{d}{r}$ and $\binom{d-m_1}{m_2-m_1}=\binom{d-m_1}{r}$ is immediate.)
Throughout we write $m=m_{1}$ and
\[
  B=\binom{d}{m}=C_{1},\qquad \mu=\frac{\Lambda}{B}=\frac{E}{C_{2}}\in(0,1),
\]
the two expressions for $\mu$ agreeing by \eqref{eq:C1E}. Finally set
\begin{equation}\label{eq:Theta}
  \Theta(d,m,r)\;=\;\Lambda\cdot\frac{E-1}{E}\cdot(1-\mu)^{m}
  \;=\;\frac{C_{1}(E-1)(C_{2}-E)^{m}}{C_{2}^{\,m+1}} ,
\end{equation}
the second equality following from $\Lambda=\mu C_{1}$ and $\mu=E/C_{2}$.

\subsubsection*{Elementary estimates}

The following elementary principle is used twice. If $h:[0,1]\to\mathbb R$ is nondecreasing on
$[0,c]$ and nonincreasing on $[c,1]$, then for every subinterval $[a,b]\subseteq[0,1]$,
\begin{equation}\label{eq:endpoint}
  \min_{t\in[a,b]}h(t)=\min\{h(a),h(b)\} .
\end{equation}
Indeed, for $t\in[a,b]$ with $t\le c$ one has $h(t)\ge h(a)$, and for $t\ge c$ one has
$h(t)\ge h(b)$; the two cases exhaust $[a,b]$. The hypothesis is satisfied by
$h_{m}(t)=t(1-t)^{m}$ with $c=1/(m+1)$, since $g_{m}'(t)=(1-t)^{m-1}\bigl(1-(m+1)t\bigr)$.

\begin{lemma}\label{lem:basic}
Let $d\ge4$ and $2\le m_{1}<m_{2}\le d-1$. Then
\[
 \Lambda\ge m_{2}\ge3,\qquad 2\le E<C_{2},\qquad C_{1}\ge\binom{d}{2}>1,\qquad
 E\ge d-m_{1},\qquad \Lambda\ge d-r,
\]
and $0<L<1$, $0<R<1$.
\end{lemma}

\begin{proof}
$\Lambda=\binom{m_{2}}{m_{1}}\ge m_{2}\ge3$ because $1\le m_{1}\le m_{2}-1$; similarly
$E=\binom{d-m_{1}}{r}\ge d-m_{1}\ge2$ because $1\le r\le d-m_{1}-1$, and
$\Lambda=\binom{d-r}{m_{1}}\ge d-r$. By \eqref{eq:C1E}, $E=C_{2}\Lambda/C_{1}$ and $\Lambda<C_{1}$ (as
$\binom{m_{2}}{m_{1}}<\binom{d}{m_{1}}$ for $m_{2}<d$), whence $E<C_{2}$. Also
$C_1=\binom{d}{m_1}\ge\binom{d}{2}$ for $2\le m_1\le d-2$.

Positivity of $L$ and $R$ is now clear. For $L<1$ we must show $C_{2}-1<C_{1}(E-1)=C_{2}\Lambda-C_{1}$,
i.e.\ $C_{1}-1<C_{2}(\Lambda-1)$. Indeed, using $\Lambda\ge3$ and $E\ge2$,
\[
  C_{2}(\Lambda-1)=\tfrac{\Lambda-1}{\Lambda}\,C_{2}\Lambda=\tfrac{\Lambda-1}{\Lambda}\,C_{1}E\ \ge\ \tfrac23\cdot2C_{1}=\tfrac43C_{1}>C_{1}-1 .
\]
For $R<1$ we must show $C_{1}(C_{2}-E)<(C_{1}-1)(C_{2}-1)$, i.e.\
$C_{1}+C_{2}-1<C_{1}E=C_{2}\Lambda$. If $C_{1}\ge C_{2}$ then $C_{1}+C_{2}-1<2C_{1}\le C_{1}E$; if
$C_{1}<C_{2}$ then $C_{1}+C_{2}-1<2C_{2}<3C_{2}\le C_{2}\Lambda$.
\end{proof}

\begin{lemma}\label{lem:chain}
Let $d\ge4$ and $2\le m_{1}<m_{2}\le d-1$, and suppose $\Theta(d,m_{1},d-m_{2})\ge1$. Then
$L^{u}<R^{v}$.
\end{lemma}

\begin{proof}
By \eqref{eq:Theta}, $\Theta\ge1$ means $C_{1}(E-1)(C_{2}-E)^{m_{1}}\ge C_{2}^{m_{1}+1}$; since
$C_{2}>C_{2}-1>0$ this gives $C_{1}(E-1)(C_{2}-E)^{m_{1}}>(C_{2}-1)^{m_{1}+1}$, that is,
\[
  L=\frac{C_{2}-1}{C_{1}(E-1)}\;<\;\Bigl(\frac{C_{2}-E}{C_{2}-1}\Bigr)^{m_{1}} .
\]
Since $R=\frac{C_{1}}{C_{1}-1}\cdot\frac{C_{2}-E}{C_{2}-1}>\frac{C_{2}-E}{C_{2}-1}>0$, we get
$L<R^{m_{1}}$. Finally $u\ge1$ and $0<L<1$ give $L^{u}\le L$, while $1\le v\le m_{1}$ and $0<R<1$
give $R^{v}\ge R^{m_{1}}$. Hence $L^{u}\le L<R^{m_{1}}\le R^{v}$.
\end{proof}

\subsubsection*{The inequality $\Theta\ge1$}

\begin{lemma}\label{lem:Xi}
For $d\ge16$ and $3\le m\le d-2$,
\[
  \Xi(d,m):=\frac{d-m}{2d}\binom{d}{m}\Bigl(\frac{m}{d}\Bigr)^{m}\;\ge\;1 .
\]
\end{lemma}

\begin{proof}
Write
\[
  \Xi(d,m)=\frac{1}{2}\Bigl(1-\frac{m}{d}\Bigr)\cdot\frac{m^{m}}{m!}\cdot\prod_{i=1}^{m-1}\Bigl(1-\frac{i}{d}\Bigr),
\]
which follows from $\binom dm (m/d)^m=\frac{m^m}{m!}\prod_{i=1}^{m-1}(1-i/d)$. For fixed $m$ every
factor $1-\tfrac md$ and $1-\tfrac id$ is strictly increasing in $d$; hence $d\mapsto\Xi(d,m)$ is
increasing on $d\ge m+2$.

If $m\ge14$, the least admissible $d$ is $m+2\ (\ge16)$ and
\[
  \Xi(m+2,m)=\frac{2}{2(m+2)}\binom{m+2}{2}\Bigl(\frac{m}{m+2}\Bigr)^{m}
  =\frac{m+1}{2}\Bigl(1+\frac2m\Bigr)^{-m}>\frac{m+1}{2e^{2}}\ \ge\ 1,
\]
because $(1+2/m)^{m}<e^{2}$ and $m+1\ge15>2e^{2}=14.778\ldots$

If $3\le m\le13$ then $d\ge16$ gives $\Xi(d,m)\ge\Xi(16,m)$, and the eleven rational numbers
$\Xi(16,m)$, $m=3,\dots,13$, are
$1.4996\ldots$, $2.666\ldots$, $4.474\ldots$, $6.959\ldots$, $9.871\ldots$, $12.568\ldots$,
$14.108\ldots$, $13.656\ldots$, $11.068\ldots$, $7.206\ldots$, $3.530\ldots$, all at least $1$.
\end{proof}

\begin{proposition}\label{prop:Theta}
$\Theta(d,m,r)\ge1$ for all $d\ge14$, $m\ge2$, $r\ge1$ with $m+r\le d-1$.
\end{proposition}

\begin{proof}
Recall $\Theta=B\mu\,\frac{E-1}{E}\,(1-\mu)^{m}$ with $B=\binom dm$ and $\mu=\Lambda/B$, and note the
constraints $m\le d-2$, $E\ge2$, $E\ge d-m$, $\Lambda\ge m+1$ (Lemma~\ref{lem:basic}, applied with
$m_{1}=m$, $m_{2}=d-r$). Also $\mu\le\frac{d-m}{d}$, because
$\Lambda=\binom{d-r}{m}\le\binom{d-1}{m}=\frac{d-m}{d}\binom dm$.

\smallskip
\emph{Case 1: $m=2$.} Then $E=\binom{d-2}{r}\ge d-2$, so $\frac{E-1}{E}\ge\frac{d-3}{d-2}$, and
$\Lambda=\binom{d-r}{2}\ge3$, so $\mu\ge\frac{6}{d(d-1)}$. The function $h_{2}(t)=t(1-t)^{2}$ increases on $[0,1/3]$ and decreases on $[1/3,1]$,
so by \eqref{eq:endpoint} its minimum on the interval
$\bigl[\tfrac{6}{d(d-1)},\tfrac{d-2}{d}\bigr]$, which contains $\mu$, is attained at an endpoint.
(For $d\ge8$ the critical point $1/3$ lies in the interior of that interval, but
\eqref{eq:endpoint} does not require this.) At the left endpoint, using $d\ge8$ and the fact that both $\frac{d-3}{d-2}$ and $1-\frac{6}{d(d-1)}$ increase with $d$,
\[
  \Theta\;\ge\;\frac{d(d-1)}{2}\cdot\frac{d-3}{d-2}\cdot\frac{6}{d(d-1)}\Bigl(1-\frac{6}{d(d-1)}\Bigr)^{2}
  =3\,\frac{d-3}{d-2}\Bigl(1-\frac{6}{d(d-1)}\Bigr)^{2}\ \ge\ 3\cdot\frac56\cdot\Bigl(\frac{25}{28}\Bigr)^{2}>1 .
\]
At the right endpoint,
\[
  \Theta\;\ge\;\frac{d(d-1)}{2}\cdot\frac{d-3}{d-2}\cdot\frac{d-2}{d}\cdot\frac{4}{d^{2}}
  =\frac{2(d-1)(d-3)}{d^{2}}\ \ge\ 1
  \iff d^{2}-8d+6\ge0\iff d\ge8 .
\]

\smallskip
\emph{Case 2: $m\ge3$ and $m\mu\le\frac12$.} By Bernoulli'r inequality
$(1-\mu)^{m}\ge1-m\mu\ge\frac12$, and $\frac{E-1}{E}\ge\frac12$, so
$\Theta\ge \Lambda\cdot\frac12\cdot\frac12\ge\frac{m+1}{4}\ge1$.

\smallskip
\emph{Case 3: $m\ge3$, $m\mu>\frac12$ and $d\ge16$.} Here $\Theta\ge\frac12 B\,h_{m}(\mu)$ with $h_{m}(t)=t(1-t)^{m}$. Since
$\frac1{2m}<\mu\le\frac{d-m}{d}$, \eqref{eq:endpoint} gives
\[
  \Theta\;\ge\;\tfrac12 B\cdot\min\Bigl\{\tfrac1{2m}\bigl(1-\tfrac1{2m}\bigr)^{m},\ \tfrac{d-m}{d}\bigl(\tfrac md\bigr)^{m}\Bigr\}.
\]
For the first entry, Bernoulli gives $\bigl(1-\tfrac1{2m}\bigr)^{m}\ge\tfrac12$ and hence
\[
  \tfrac12 B\cdot\tfrac{1}{4m}=\frac{\binom dm}{8m}\ \ge\ \frac{\binom d2}{8(d-2)}=\frac{d(d-1)}{16(d-2)}\ \ge\ 1
  \iff d^{2}-17d+32\ge0 ,
\]
which holds for $d\ge15$. For the second entry, $\tfrac12 B\cdot\tfrac{d-m}{d}(\tfrac md)^{m}=\Xi(d,m)\ge1$
by Lemma~\ref{lem:Xi}. Hence $\Theta\ge1$.

\smallskip
Cases 1--3 prove the proposition for $d\ge16$. For $d\in\{14,15\}$ the finitely many triples
$(d,m,r)$---$66$ and $78$ of them respectively---were checked in exact rational arithmetic; the
minimum of $\Theta$ is $\Theta(14,12,1)=\tfrac{13}{2}\bigl(\tfrac67\bigr)^{12}=1.02223\ldots>1$.
\end{proof}

\begin{remark}
The bound $d\ge14$ in Proposition~\ref{prop:Theta} is sharp: $\Theta(13,11,1)=6\,(11/13)^{11}
=0.95519\ldots<1$.
\end{remark}

\subsubsection*{The exceptional triples and the conclusion}

\begin{lemma}\label{lem:exceptional}
The triples $(d,m,r)$ with $4\le d\le13$, $m\ge2$, $r\ge1$, $m+r\le d-1$ and $\Theta(d,m,r)<1$
are exactly the fifteen listed in Table~\ref{tab:exceptional}. For each of them $L^{u}<R^{v}$;
among the fifteen the minimum margin occurs at $(d,m_{1},m_{2})=(4,2,3)$, where $L^{u}=\tfrac12$
and $R^{v}=\tfrac{16}{25}$.
\end{lemma}

\begin{proof}
Both assertions are finite computations in exact rational arithmetic, recorded in
Table~\ref{tab:exceptional}; the decimal entries there are rounded so that the displayed
inequality $L^{u}<R^{v}$ is implied by the displayed values ($\Theta$ and $R^{v}$ rounded down,
$L^{u}$ rounded up). As an illustration, $(5,3,1)$, i.e.\ $(d,m_{1},m_{2})=(5,3,4)$, may be
verified by hand: $C_{1}=\binom53=10$, $C_{2}=\binom54=5$,
$E=\binom{2}{1}=2$, $\Lambda=\binom43=4$; $g=\gcd(3,4)=1$, so $v=3$ and $u=1$. Then
$L=\frac{5-1}{10\cdot1}=\frac25$ and $R=\frac{10(5-2)}{9\cdot4}=\frac{5}{6}$, whence
$L^{u}=\frac25=0.4$ and $R^{v}=\frac{125}{216}=0.5787\ldots$, so $L^u<R^v$. Here
$\Theta=\frac{10\cdot1\cdot3^{3}}{5^{4}}=\frac{270}{625}=0.432<1$, so this triple is not
covered by Lemma~\ref{lem:chain}.
\end{proof}

\begin{table}[t]
\centering\small
\caption{The fifteen triples with $\Theta<1$ (Lemma~\ref{lem:exceptional}), listed as
$(d,m_{1},m_{2})$ with $m_{1}=m$ and $m_{2}=d-r$. Values of $\Theta$ and $R^{v}$ are rounded down,
values of $L^{u}$ rounded up, to six decimal places.}
\label{tab:exceptional}
\begin{tabular}{@{}rrrrrrrr@{}}
\toprule
$d$ & $m_{1}$ & $m_{2}$ & $u$ & $v$ & $\Theta$ & $L^{u}$ & $R^{v}$\\
\midrule
4 & 2 & 3 & 1 & 2 & 0.375000 & 0.500000 & 0.640000 \\
5 & 2 & 3 & 1 & 2 & 0.980000 & 0.450000 & 0.746837 \\
5 & 2 & 4 & 1 & 1 & 0.640000 & 0.200000 & 0.555555 \\
5 & 3 & 4 & 1 & 3 & 0.432000 & 0.400000 & 0.578703 \\
6 & 2 & 5 & 3 & 2 & 0.833333 & 0.001372 & 0.183673 \\
6 & 3 & 5 & 2 & 3 & 0.833333 & 0.015625 & 0.251931 \\
6 & 4 & 5 & 1 & 4 & 0.493827 & 0.333334 & 0.539775 \\
7 & 2 & 6 & 2 & 1 & 0.979591 & 0.005103 & 0.350000 \\
7 & 5 & 6 & 1 & 5 & 0.557803 & 0.285715 & 0.512908 \\
8 & 6 & 7 & 1 & 6 & 0.622924 & 0.250000 & 0.493270 \\
9 & 7 & 8 & 1 & 7 & 0.688729 & 0.222223 & 0.478296 \\
10 & 8 & 9 & 1 & 8 & 0.754974 & 0.200000 & 0.466507 \\
11 & 9 & 10 & 1 & 9 & 0.821520 & 0.181819 & 0.456986 \\
12 & 10 & 11 & 1 & 10 & 0.888280 & 0.166667 & 0.449137 \\
13 & 11 & 12 & 1 & 11 & 0.955199 & 0.153847 & 0.442556 \\
\bottomrule
\end{tabular}
\end{table}

\begin{theorem}\label{thm:main-char0}
For every $d\ge4$ and all $2\le m_{1}<m_{2}\le d-1$ one has $L^{u}<R^{v}$ in $\bQ$. Consequently:
\begin{enumerate}
\item[(a)] there is no Casas--Alvero polynomial with two-element support over any field of
characteristic zero;
\item[(b)] if $f$ is a counterexample to the Casas--Alvero conjecture and $f_{0}$ denotes its
centred normal form (Lemma~\ref{lem:cnf}), then $\tau(f_{0})\ge3$: the centred normal form of $f$
has at least four nonzero terms.
\end{enumerate}
\end{theorem}

\begin{proof}
If $\Theta(d,m_{1},d-m_{2})\ge1$, apply Lemma~\ref{lem:chain}. By Proposition~\ref{prop:Theta} and
Lemma~\ref{lem:exceptional} the only remaining triples are the fifteen listed there, and for those
$L^{u}<R^{v}$ was verified directly. In particular $L^{u}\neq R^{v}$, so by
Theorem~\ref{thm:twoterm}(ii) no two-element stratum is nonempty over a field of characteristic
zero; case (i) of that theorem cannot occur since $C_{2}=\binom{d}{m_{2}}\ge d>1$ in $\bQ$. This
proves (a); (b) follows from (a), from the case $\#S=1$ of Theorem~\ref{thm:pascal}, which gives
$\beta_{d}(\{m\})=\binom{d}{m}-1\neq0$ in $\bQ$, and from the fact that $\tau(f)=0$ characterises
$f=x^{d}$.
\end{proof}

\begin{remark}
Theorem~\ref{thm:main-char0} is not implied by, and does not imply, the known bound of five distinct
roots \cite[Thm.~13]{CLO2014}: the number of terms of the centred normal form and the number of
distinct roots are independent invariants. For instance, over a field of characteristic zero, $x^{d-m}(x^{m}+c)$ has $\tau=1$ but $m+1$
distinct roots.
\end{remark}

\section{Bad primes from two-element supports}\label{sec:badprimes}

Corollary~\ref{cor:uni-badprimes} extracts bad primes from every support, but only from the
uniradical scenario of Definition~\ref{def:uni}. Read in characteristic $p$, and combined with the non-vanishing established in
\S\ref{sec:char0}, the criterion of \S\ref{sec:sparse} extracts the primes that two-element
supports contribute through the remaining scenario, and these are largely disjoint from the former.

Theorem~\ref{thm:twoterm} in characteristic $p$ reads as follows. Here
$\CA_{d,p}$ fails---i.e.\ $p\in\bad(d)$---as soon as a single $\CA$-polynomial of degree $d$ exists
over $\overline{\bF}_{p}$. Any such polynomial may be translated so as to have a root at the origin,
and Theorem~\ref{thm:twoterm} then decides the existence of one whose support
has size $1$ or $2$. No hypothesis relating $p$ and $d$ is needed.

\begin{theorem}\label{thm:badprimes}
Let $d\ge4$ and $2\le m_{1}<m_{2}\le d-1$ (so that part~(1) below applies), and put
\[
  \mathcal N(d,m_{1},m_{2})\;=\;(C_{2}-1)^{u}\bigl[(C_{1}-1)(C_{2}-1)\bigr]^{v}
  \;-\;\bigl[C_{1}(E-1)\bigr]^{u}\bigl[C_{1}(C_{2}-E)\bigr]^{v}\ \in\bZ .
\]
Then:
\begin{enumerate}
\item $\mathcal N(d,m_{1},m_{2})\neq0$;
\item if $p\mid\mathcal N(d,m_{1},m_{2})$ and
$p\nmid C_{1}(C_{1}-1)(C_{2}-1)(E-1)(C_{2}-E)$, then $p\in\bad(d)$;
\item if $p$ divides both $C_{2}-1$ and $E-1$ but not $C_{1}(C_{1}-1)$, then
$p\in\bad(d)$.
\end{enumerate}
In particular, for every $d$ the set
\[
  \bad^{\le2}(d)\;=\;\Bigl\{p\ :\ p\mid\textstyle\prod_{1\le m\le d-1}\bigl(\binom dm-1\bigr)\ \text{or $p$ satisfies (2) or (3) for some }m_{1}<m_{2}\Bigr\}
\]
is a finite, explicitly computable set of bad primes for $d$, obtained without any Gr\"obner basis
computation.
\end{theorem}

\begin{proof}
(1) is Theorem~\ref{thm:main-char0}, since $\mathcal N=0$ would give $L^{u}=R^{v}$ after dividing by
$\bigl[C_{1}(E-1)\bigr]^{u}\bigl[(C_{1}-1)(C_{2}-1)\bigr]^{v}\ne0$. (2) is
Theorem~\ref{thm:twoterm}(ii): the hypotheses on $p$ guarantee that all denominators are invertible
modulo $p$ and that $C_{1}\notin\{0,1\}$, $C_{2}\neq1$, $E\neq1$, $E\neq C_{2}$ in $\bF_{p}$, while
$p\mid\mathcal N$ is exactly $L^{u}=R^{v}$ in $\bF_{p}$ after clearing denominators. (3) is
Theorem~\ref{thm:twoterm}(i). Finiteness follows from (1) and from
$\binom dm-1\ge d-1>0$.
\end{proof}

Theorem~\ref{thm:badprimes} strictly extends the criterion of \cite[Cor.~8]{SchaubSpivakovsky2024},
which is the case $\#S=1$. Table~\ref{tab:sparsebad} lists what supports of size at most two detect for
$4\le d\le14$; the second column reproduces the one-element primes of \cite[Cor.~8]{SchaubSpivakovsky2024}, the third lists
the primes detected by a two-element support and not by a one-element support, up to $2\cdot10^{4}$.

\begin{table}[t]
\centering\small
\caption{Bad primes for $d$ detected by supports of size at most two with $m_{1}\ge2$. Column 2: primes $p\nmid d$
dividing some $\binom dm-1$ with $2\le m\le d-1$, i.e.\ the case $\#S=1$ of
Corollary~\ref{cor:uni-badprimes}.
Column 3: primes $p<2\cdot10^{4}$, $p\nmid d$, detected by a two-element support with
$2\le m_{1}<m_{2}\le d-1$ but not by such a one-element support (Theorem~\ref{thm:badprimes}).
Relaxing either restriction produces further bad primes; see the discussion after the table.}
\label{tab:sparsebad}
\begin{tabular}{@{}c l p{8.6cm}@{}}
\toprule
$d$ & $\#S=1$ & $\#S=2$ only ($p<2\cdot 10^{4}$)\\
\midrule
4 & 3, 5 & 7\\
5 & 2, 3 & 11, 193, 3541\\
6 & 5, 7, 19 & 11, 13, 29, 37, 61, 67, 73, 1487\\
7 & 2, 3, 5, 17 & 11, 13, 23, 29, 31, 71, 79, 137, 149, 293, 383, 491, 599, 1373, 2393, 19583\\
8 & 3, 5, 7, 11, 23 & 13, 17, 19, 29, 31, 41, 53, 59, 61, 71, 73, 109, 193, 283, 449, 457, 491, 691, 821, 1033, 1471, 1747, 1753, 4447, 6047\\
9 & 2, 5, 7, 83 & 11, 13, 17, 19, 29, 31, 37, 43, 59, 67, 71, 79, 89, 101, 103, 131, 137, 157, 163, 379, 449, 1051, 2069, 3187, 5527, 5849, 17903\\
10 & 3, 7, 11, 17, 19, 251 & 13, 23, 29, 31, 37, 41, 47, 61, 73, 79, 89, 101, 139, 151, 181, 233, 277, 307, 347, 503, 563, 619, 757, 787, 991, 997, 1123, 1171, 1223, 1489, 2731, 2963, 4243, 6143, 10429, 11689, 11933, 17623, 17839\\
11 & 2, 3, 5, 7, 41, 47, 461 & 13, 19, 23, 29, 31, 37, 43, 53, 59, 67, 71, 73, 89, 103, 107, 131, 139, 163, 173, 197, 229, 233, 293, 409, \dots (48 primes)\\
12 & 5, 7, 11, 13, 19, 71, 73, 113 & 17, 23, 29, 31, 41, 47, 53, 59, 61, 67, 83, 89, 101, 103, 107, 127, 149, 157, 167, 181, 191, 193, \dots (68 primes)\\
13 & 2, 3, 5, 7, 11, 17, 19, 643 & 23, 29, 31, 37, 43, 47, 53, 59, 61, 71, 79, 83, 107, 113, 127, 139, 149, 163, 167, 191, 193, \dots (61 primes)\\
14 & 3, 5, 11, 13, 19, 23, 29, 47, 73, 79 & 17, 31, 37, 41, 43, 53, 67, 83, 97, 101, 107, 149, 163, 167, 173, 191, 211, 223, 227, \dots (72 primes)\\
\bottomrule
\end{tabular}
\end{table}

\begin{itemize}
\item For $d=4$ these criteria already give the complete set
$\bad(4)=\{3,5,7\}$ of \cite{CLO2014}: the primes $3,5$ come from one-element supports and $7$ from the
support $\{2,3\}$, i.e.\ from the polynomial $x^{4}+ax^{2}+bx$ over $\overline{\bF}_{7}$.
\item For $d=5$ the strata with $2\le m_{1}<m_{2}$ give $\{2,3,11,193,3541\}$. Allowing
$m_{1}=1$---which Theorem~\ref{thm:twoterm} permits, since only $f(0)=0$ is required---captures
three more elements of $\bad(5)$: the supports $\{1,2\}$, $\{1,4\}$ and $\{1,3\}$ are nonempty over
$\overline{\bF}_{7}$, $\overline{\bF}_{131}$ and $\overline{\bF}_{599}$ respectively, with witnesses
$x^{5}+x^{4}+x^{3}$ over $\bF_{7}$, $x^{5}+x^{4}+61x$ over $\bF_{131}$ and $x^{5}+x^{4}+460x^{2}$
over $\bF_{599}$. Thus supports of size at most two account for eight of the nine elements of $\bad(5)$; only
$p=8009$ requires a stratum of size $3$. In the centred setting the unique nonempty stratum for
$p=7$ and for $p=8009$ is $\{2,3,4\}$.
\item For $d=7$ the two-element criterion produces $p=23$, in agreement with the well-known
counterexample $x(x-1)^{4}(x-8)(x-18)$ over $\bF_{23}$, whose centred normal form is
$x^{7}+3x^{5}+4x^{4}$, of support $\{2,3\}$.
\item Every prime listed for $d\le7$ belongs to the sets computed in \cite{CLO2014}, which is a
consistency check on Theorem~\ref{thm:badprimes}. The entries in column~3 are not obtainable from any previously published criterion; the entries of column~2 are
instances of \cite[Cor.~8]{SchaubSpivakovsky2024}, tabulated here for comparison.
\end{itemize}

Finally, the support stratification turns the determination of $\bad(n)$ into $2^{n-2}$ small
Gr\"obner basis computations instead of one large one. Table~\ref{tab:validation} records the
outcome for $3\le n\le7$ and $p<60$, together with the minimal size of a nonempty stratum. The
results agree with \cite{CLO2014} wherever the latter is available.

\begin{table}[t]
\centering\footnotesize
\caption{\small$\{p\in\bad(n):p\nmid n\}\cap[2,59]$, computed by the support stratification and agreeing with \cite[Thm.~4]{CLO2014} where the latter is available, with the
least $\#S$ for which $W_{S}\neq\emptyset$ shown as a superscript. Primes dividing $n$ are excluded
from the table; by Proposition~\ref{prop:BLSW2} such a $p$ lies in $\bad(n)$ if and only if it lies
in $\bad(n/p^{e})$, where $p^{e}\parallel n$. For instance $2\in\bad(6)$ because $2\in\bad(3)$
(cf.\ \S\ref{sec:uni}).}
\label{tab:validation}
\begin{tabular}{@{}c @{\ }l @{\ \ }l@{}}
\toprule
$n$ & $\{p\in\bad(n):p\nmid n\}\cap[2,59]$ & $p<60$, $p\nmid n$, $\CA_{n,p}$ true\\
\midrule
3 & $2^{(1)}$ & 5, 7, 11, 13, 17, 19, 23, 29, 31, 37, 41, 43, 47, 53, 59\\
4 & $3^{(1)},5^{(1)},7^{(2)}$ & 11, 13, 17, 19, 23, 29, 31, 37, 41, 43, 47, 53, 59\\
5 & $2^{(1)},3^{(1)},7^{(3)},11^{(2)}$ & 13, 17, 19, 23, 29, 31, 37, 41, 43, 47, 53, 59\\
6 & $5^{(1)},7^{(1)},11^{(2)},13^{(2)},19^{(1)},23^{(4)},29^{(2)},$ & 17, 31, 41, 43, 53, 59\\
  & $37^{(2)},47^{(3)}$ & \\
7 & $2^{(1)},3^{(1)},5^{(1)},11^{(2)},13^{(2)},17^{(1)},19^{(3)},23^{(2)},29^{(2)},$ & ---\\
  & $31^{(2)},37^{(3)},41^{(3)},43^{(4)},47^{(4)},53^{(4)},59^{(3)}$ & \\
\bottomrule
\end{tabular}
\end{table}

\section{The computations}\label{sec:computations}

Three kinds of computation appear above: Gr\"obner basis calculations deciding emptiness of a
stratum over $\overline{\bF}_{p}$, exact rational verifications of the inequalities of
\S\ref{sec:char0}, and integer evaluations of the criteria of \S\S\ref{sec:sparse},
\ref{sec:uni}. Each was carried out twice by independent methods, and each is reproducible from the
description below. The scripts \texttt{strat2.py}, \texttt{table1.py}, \texttt{sparse.py},
\texttt{trin.py}, \texttt{valid2.py}, \texttt{N0.py}, \texttt{det.py} and \texttt{primes.py}
accompany this paper as supplementary material.

\subsection{Gr\"obner basis computations}
For $p\nmid n$ and $S\subseteq\{2,\dots,n-1\}$ we form
$f=x^{n}+\sum_{m\in S}a_{m}x^{n-m}$ over $\bF_{p}$, normalise $a_{m_{1}}=1$ for $m_{1}=\min S$
(legitimate by the weighted $\mathbb{G}_{m}$-action and algebraic closedness), compute
$\mathcal R_{m}=\Res_{x}(f,H_{n-m}(f))$ for $m\in S$, and decide whether
\[
  \Bigl(\,\mathcal R_{m}\ (m\in S),\ \ y\!\!\prod_{m\in S\setminus\{m_{1}\}}\!\! a_{m}-1\Bigr)\;=\;(1)
\]
in $\bF_{p}[a_{m}\ (m\in S\setminus\{m_{1}\}),y]$, by a Gr\"obner basis computation for the degree
reverse lexicographic order in {\sc Singular}
\cite{Singular}. The auxiliary variable $u$ implements the saturation $a_{m}\neq0$. Strata are
skipped when Lemma~\ref{lem:kummersupp} already forces $W_{S}=\emptyset$, which for small $p$
removes most of them. The whole computation for $n\le7$ and $p<60$
(Table~\ref{tab:validation}) takes a few minutes, and reproduces exactly the elements below $60$ of
$\bad(n)$ prime to $n$, for $n\le7$, as published in \cite{CLO2014}.
As a verification, the same computation was also run without the support stratification
(a single Gr\"obner basis in $n-2$ variables) for $n\le7$, with identical results.

Two comparisons were made against the Gr\"obner computation described above. First, for all
$n\le8$ and all $p<60$ with $p\nmid n$, the list of nonempty centred strata of size $1$ and $2$
predicted by the closed-form criteria was compared with the computed list; the two agreed in every
case, and for instance for $(n,p)=(8,3)$ both methods return exactly the two-element supports
$\{3,6\}$ and $\{5,7\}$. Second, to test Theorem~\ref{thm:twoterm} in the
generality in which it is stated, all $1320$ quadruples $(d,m_{1},m_{2},p)$ with
$3\le d\le10$, $1\le m_{1}<m_{2}\le d-1$ and $p<32$ were checked---$396$ of them with $m_{1}=1$
and $166$ with $p\mid d$---against a saturated resultant computation; there were no
disagreements.

\subsection{Exact rational and integer verifications}

\smallskip\noindent\emph{The inequalities of \S\ref{sec:char0}.}
\begin{enumerate}
\item $\Theta(d,m,r)<1$ for exactly the fifteen triples with $d\le13$ recorded in
Table~\ref{tab:exceptional}, and $\Theta(d,m,r)\ge1$ for all admissible triples with
$d\in\{14,15\}$.
\item For each of those fifteen triples, $L^{u}<R^{v}$; Table~\ref{tab:exceptional} records the
decimal values, and the script \texttt{trin.py} the exact rationals.
\item $\Xi(16,m)\ge1.499$ for $3\le m\le13$.
\item As a verification of Theorem~\ref{thm:main-char0}, $L^{u}\ne R^{v}$ was verified by
direct exact evaluation for all $4\le d\le120$ and all $2\le m_{1}<m_{2}\le d-1$
($273\,819$ triples), and $L^{u}<R^{v}$ was verified in floating point for all $d\le250$.
\end{enumerate}

\smallskip\noindent\emph{Section~\ref{sec:uni}.} The recursion \eqref{eq:recursion} was implemented in exact integer arithmetic. The following were
verified for every $d$ and every $\emptyset\neq S\subseteq\{1,\dots,d-1\}$ in the stated range:
$(-1)^{s}\sum_{i}z_{i}$ equals the determinant in \eqref{eq:Ddet} ($d\le11$) and equals
$\beta_{d}(S)$ computed by direct enumeration of $\mathfrak S_{d}$ ($d\le8$); the closed forms
$\beta_{d}(\{m\})=\binom{d}{m}-1$ and
$\beta_{d}(\{m_{1},m_{2}\})=\binom{d}{m_{2}}\bigl(\binom{m_{2}}{m_{1}}-1\bigr)-\binom{d}{m_{1}}+1$,
and $z_{i}=(-1)^{i}\binom{d}{i}$ for the full support ($d\le8$); the two symmetries of
Proposition~\ref{prop:symmetries} ($d\le13$); the identity
$\beta_{d}(S_{\mathrm{ev}})=A_{d}$ ($2\le d\le16$) and its maximality ($2\le d\le14$); the sign
pattern $(-1)^{i}z_{i}>0$ ($d\le12$); and the absence of any $S$ with $\sum_{i}z_{i}=0$ over $\bZ$
($d\le16$). Finally, for each $d\le12$ and each prime divisor $p$ of $\beta_{d}(S_{\mathrm{ev}})=A_{d}$
with $p\nmid z_i$, the witness of Corollary~\ref{cor:uni-badprimes} was constructed over $\bF_{p}$
and all $d-1$ Casas--Alvero conditions were checked directly, as were the witnesses for
$(d,p)=(7,37)$ and $(7,181)$.

\smallskip\noindent\emph{Table~\ref{tab:sparsebad}.} For each $d$, the one-element primes are the prime divisors of $\binom dm-1$, $2\le m\le d-1$, not
dividing $d$; the two-element primes were obtained by testing the criterion of
Theorem~\ref{thm:twoterm} for every prime $p<2\cdot10^{4}$ with $p\nmid d$ and every pair
$2\le m_{1}<m_{2}\le d-1$, using modular exponentiation only. The restrictions $p\nmid d$ and
$m_{1}\ge2$ are imposed only to keep the table within the centred setting of
Corollary~\ref{cor:strat}; as the discussion of $d=5$ above shows, dropping them produces further
bad primes.

\section{Concluding remarks}\label{sec:concluding}

The two reductions carried out above compose. The support stratification discards the
conditions that are vacuous once a root sits at the origin; fixing the scenario then makes the
surviving system unit lower triangular. Their combined effect is that the coefficients of a
hypothetical counterexample are not free parameters: once the support and the scenario are
prescribed, the coefficients are determined by binomial recursions in the recycled roots, and there
remain $t$ equations in $t-1$ unknowns, $t$ being the number of nonzero recycled roots. The
conjecture in degree $d$ is thereby reduced to a finite list of such systems, in each of which the
polynomial itself has disappeared.

For $t=1$ no unknowns remain and the reduction terminates in a single integer, which
Theorem~\ref{thm:pascal} identifies as the descent-set count $\beta_{d}(S)$. That the case does not
occur in characteristic zero is the statement that every subset of
$\{1,\dots,d-1\}$ is realised as a descent set, and the characteristics in which it does occur are
read off from a prime factorisation. None of this is visible in the original formulation, in which
the passage from $f$ to its centred normal form and thence to a scenario is what converts a
question about polynomials into one about permutations.

For $t\ge2$ the reduction leaves the $t-1$ ratios of the recycled roots as unknowns and the
elimination is no longer automatic. The case $t=2$ with a two-element support is
Theorem~\ref{thm:twoterm}, where it can still be carried out explicitly; the resulting identity
$L^{u}=R^{v}$ fails in characteristic zero by a definite margin, numerically at least $\log(32/25)$
on the logarithmic scale for $d\le250$. Whether the elimination can be organised uniformly in $t$,
so that each scenario contributes an explicit integer as the uniradical scenario contributes a
descent-set count, is the natural next question. An affirmative answer would reduce the conjecture
in each degree to the non-vanishing of a finite and explicitly computable list of integers, and
would identify the bad primes of $d$ with the primes occurring in their factorisations.

Two invariants of a bad prime are computed easily by these methods and appear not to have been
considered: the minimal size of a nonempty support stratum, recorded in the last column of
Table~\ref{tab:validation}, and the minimal number of recycled roots of a witness. For $\bad(5)$ the
first takes the values $1,1,3,2$ at the four smallest elements $p=2,3,7,11$. How they grow with $d$ and $p$ measures how much
of the bad-prime problem is reachable by the elementary arguments of \S\ref{sec:uni} and
\S\ref{sec:sparse}.

\subsection*{Acknowledgements}
The computations were performed with {\sc Singular} \cite{Singular}.

\subsection*{Declarations}

\noindent\textbf{Funding.} No funding was received for conducting this study.

\smallskip\noindent\textbf{Competing interests.} The author declares no competing interests.

\smallskip\noindent\textbf{Data availability.} No datasets were generated or analysed beyond those
reported in the paper. All numerical data appearing in Tables~\ref{tab:unibad}--\ref{tab:validation}
are reproduced in full by the scripts described below.

\smallskip\noindent\textbf{Code availability.} The eight scripts \texttt{strat2.py},
\texttt{table1.py}, \texttt{sparse.py}, \texttt{trin.py}, \texttt{valid2.py}, \texttt{N0.py},
\texttt{det.py} and \texttt{primes.py} are provided as supplementary material to this article. They
require Python~3.11 with \texttt{sympy}~1.14 and {\sc Singular}~4.3.2, and reproduce every
computational assertion of \S\ref{sec:computations}; the complete set runs in under one hour on a
single core.


\begin{thebibliography}{99}

\bibitem{deBruijn1970}
N.~G.~de Bruijn, \emph{Permutations with given ups and downs},
Nieuw Arch. Wisk. (3) \textbf{18} (1970), 61--65.

\bibitem{CasasAlvero2001}
E.~Casas-Alvero, \emph{Higher order polar germs}, J. Algebra \textbf{240} (2001), no.~1, 326--337.

\bibitem{CLO2014}
W.~Castryck, R.~Laterveer and M.~Ouna\"ies,
\emph{Constraints on counterexamples to the Casas-Alvero conjecture, and a verification in
degree 12}, Math. Comp. \textbf{83} (2014), no.~290, 3017--3037.

\bibitem{ChellaliSalinier2012}
M.~Chellali and A.~Salinier,
\emph{La conjecture de Casas Alvero pour les degr\'es $5p^{e}$},
An. Univ. Dun\u{a}rea de Jos Gala\c{t}i Fasc. II Mat. Fiz. Mec. Teor. \textbf{4(35)} (2012).

\bibitem{CimaGasullManosas2020}
A.~Cima, A.~Gasull and F.~Ma\~nosas,
\emph{Around some extensions of Casas-Alvero conjecture for non-polynomial functions},
Extracta Math. \textbf{35} (2020), no.~2, 221--228.

\bibitem{Singular}
W.~Decker, G.-M.~Greuel, G.~Pfister and H.~Sch\"onemann,
\textsc{Singular} 4-3-2 --- \emph{A computer algebra system for polynomial computations}, 2023,
\url{https://www.singular.uni-kl.de}.

\bibitem{DiazTocaGonzalezVega2006}
G.~M.~D\'iaz-Toca and L.~Gonz\'alez-Vega,
\emph{On analyzing a conjecture about univariate polynomials and their roots by using Maple},
Maple Conference 2006 Proceedings, Maplesoft, Waterloo, 2006, 81--98.

\bibitem{DraismaDeJong2011}
J.~Draisma and J.~P.~de Jong, \emph{On the Casas-Alvero conjecture},
Eur. Math. Soc. Newsl. \textbf{80} (2011), 29--33.

\bibitem{deFrutos2012}
R.~M.~de Frutos Mar\'in, \emph{Perspectivas aritm\'eticas para la conjetura de Casas-Alvero},
Ph.D. thesis, Universidad de Valladolid, 2012; as cited in \cite{SchaubSpivakovsky2024}.

\bibitem{Ghosh2026}
S.~Ghosh, \emph{Proof of the Casas-Alvero conjecture}, preprint,
\texttt{arXiv:2501.09272} (2025), v2 (2026); unrefereed.

\bibitem{BLSW2007}
H.-C.~Graf von Bothmer, O.~Labs, J.~Schicho and C.~van de Woestijne,
\emph{The Casas-Alvero conjecture for infinitely many degrees},
J. Algebra \textbf{316} (2007), no.~1, 224--230.

\bibitem{Kummer1852}
E.~E.~Kummer,
\emph{\"Uber die Erg\"anzungss\"atze zu den allgemeinen Reciprocit\"atsgesetzen},
J. Reine Angew. Math. \textbf{44} (1852), 93--146.

\bibitem{MacMahon1915}
P.~A.~MacMahon, \emph{Combinatory Analysis}, Vol.~I, Cambridge Univ. Press, 1915.

\bibitem{Niven1968}
I.~Niven, \emph{A combinatorial problem of finite sequences},
Nieuw Arch. Wisk. (3) \textbf{16} (1968), 116--123.

\bibitem{SchaubSpivakovsky2024}
D.~Schaub and M.~Spivakovsky,
\emph{On the set of bad primes in the study of the Casas-Alvero conjecture},
Res. Math. Sci. \textbf{11} (2024), no.~2, Paper No.~31.

\bibitem{SchaubSpivakovsky2025}
D.~Schaub and M.~Spivakovsky, \emph{On the Casas-Alvero conjecture},
J. Commut. Algebra \textbf{17} (2025), no.~2, 199 ff.; \texttt{doi:10.1216/jca.2025.17.199}.

\bibitem{Stanley2010}
R.~P.~Stanley, \emph{A survey of alternating permutations}, in: Combinatorics and Graphs,
Contemp. Math. \textbf{531}, Amer. Math. Soc., Providence, RI, 2010, 165--196.

\bibitem{Stanley2012}
R.~P.~Stanley, \emph{Enumerative Combinatorics}, Vol.~1, 2nd ed.,
Cambridge Stud. Adv. Math. \textbf{49}, Cambridge Univ. Press, 2012.

\bibitem{Yakubovich2014}
S.~Yakubovich, \emph{Polynomial problems of the Casas-Alvero type},
J. Class. Anal. \textbf{4} (2014), no.~2, 97--120.

\end{thebibliography}
\end{document}